\documentclass[12pt, reqno]{amsart}

\usepackage{amsmath}
\usepackage{amscd}
\usepackage{mathrsfs}
\usepackage{amsfonts}
\usepackage{amssymb}
\usepackage{enumerate}
\usepackage{setspace}
\usepackage{esint}

\usepackage{latexsym}
\usepackage{enumerate}
\usepackage{mathrsfs}
\usepackage{verbatim}

\usepackage{array}

\usepackage{color}

\usepackage[hyperindex, colorlinks=true]{hyperref}

\usepackage[margin=1in]{geometry}

\newcommand{\eqn}{\begin{eqnarray}}
\newcommand{\een}{\end{eqnarray}}

\newtheorem{theorem}{Theorem}[section]
\newtheorem{corollary}[theorem]{Corollary}
\newtheorem{lemma}{Lemma}[section]
\newtheorem{proposition}[theorem]{Proposition}

\theoremstyle{definition}

\newtheorem{remark}{Remark}

\numberwithin{equation}{section}

\newcommand{\dv}{\operatorname{div}}

\newcommand*{\tran}{^{\mkern-1.5mu\mathsf{T}}}

\numberwithin{equation}{section}

\newcommand{\bR}{\mathbb{R}}

\newcommand\cL{\mathcal{L}}

\newcommand\sL{\mathscr{L}}

\providecommand{\set}[1]{\{#1\}}

\providecommand{\abs}[1]{\lvert#1\rvert}
\providecommand{\Abs}[1]{\left\lvert#1\right\rvert}

\providecommand{\norm}[1]{\lVert#1\rVert}

\providecommand{\tri}[1]{\lvert\!\lvert\!\lvert#1\rvert\!\rvert\!\rvert}

\DeclareMathOperator*{\esssup}{ess\,sup}

\begin{document}

\title[Uniqueness of solutions for Keller-Segel-Fluid model]{Uniqueness of solutions for Keller-Segel system of porous medium type coupled to fluid equations}

\author{Hantaek Bae}
\address{Department of Mathematical Sciences, Ulsan National Institute of Science and Technology (UNIST), Republic of Korea}
\email{hantaek@unist.ac.kr}

\author{Kyungkeun Kang}
\address{Department of Mathematics, Yonsei University, Seoul, Republic of Korea}
\email{kkang@yonsei.ac.kr}

\author{Seick Kim}
\address{Department of Mathematics, Yonsei University, Seoul, Republic of Korea}
\email{kimseick@yonsei.ac.kr}

\date{\today}



\begin{abstract}
We prove the uniqueness of H\"older continuous weak solutions via duality argument
and vanishing viscosity method for the Keller-Segel system of porous medium type equations
coupled to the Stokes system in dimensions three. An important step is the estimate of the Green
function of parabolic equations with lower order terms of variable coefficients, which seems to be of independent interest.
\end{abstract}

\maketitle

\section{Introduction}

In this paper, we consider a mathematical model of the dynamics of swimming bacteria \emph{Bacillus subtilis} in \cite{Tuval}, where the movement of bacteria is formulated as a form of porous medium equation. More precisely, we are concerned with the following model
\begin{subequations} \label{KS}
\begin{align}
\partial_{t}\eta+v\cdot \nabla \eta-\Delta \eta^{1+\alpha} +\nabla \cdot \left(\chi(c)\eta^{q}\nabla c\right)=0 &\quad \text{in $\mathbb{R}^{3}_{T}$},  \\
\partial_{t}c+v\cdot \nabla c-\Delta c+\kappa(c)\eta=0 &\quad \text{in $\mathbb{R}^{3}_{T}$},\\
\partial_{t}v-\Delta v +\nabla p+\eta\nabla \phi=0 &\quad \text{in $\mathbb{R}^{3}_{T}$},\\
\nabla \cdot v=0 &\quad \text{in $\mathbb{R}^{3}_{T}$},\\
\eta(0,x)=\eta_{0}(x), \ \ c(0,x)=c_{0}(x), \ \ v(0,x)=v_{0}(x) &\quad \text{in $\mathbb{R}^{3}$},
\end{align}
\end{subequations}
where $\alpha>0$ and $q\ge 1$ are given constants, and $\mathbb{R}^{3}_{T}=(0,T)\times \mathbb{R}^{3}$. Here $\eta$, $c$, $v$, and $p$ indicate biological cell density, the oxygen concentration, the fluid velocity, and the pressure, respectively. The nonnegative functions $k(c)$ and $\chi(c)$ denote the oxygen consumption rate and the chemotactic sensitivity, which are assumed to be locally bounded functions of $c$. Furthermore, the function $\phi$ is a time-independent potential function which indicates, for example, the gravitational force or centrifugal force. The system was proposed by Tuval et al. in \cite{Tuval} (see also \cite{CFKLM}) for the case that $\alpha=0$, $q=1$ and fluid equations are the Navier-Stokes system (see e.g. \cite{ckl, ckl-cpde, ckl-jkms, win2, win3} for related mathematical results).

Very recently, in \cite[Theorem 1.8-Theorem 1.10]{Chung 2}, the
existence of weak solutions and H\"older continuous weak solutions
are proved under certain assumptions on $(\alpha, q, \chi,\kappa)$
(compare to \cite{CK-2016, FLM, TW-13, win4}). It is, however,
unknown whether or not such solutions are unique. Our main
motivation is to show the uniqueness of the H\"older continuous weak
solutions of \eqref{KS}, for which we take the following simplified
model by taking $\chi =1$, $\kappa(c)=c$ and $q=1$, since the system
(\ref{KS}) is highly nonlinear and general $\chi$, $\kappa$ and
$q>1$ seem to be beyond our analysis (see Remark \ref{comment-100}).
So, we consider the following system of equations
\begin{subequations} \label{KS-PE}
\begin{align}
\partial_{t}\eta+v\cdot \nabla \eta-\Delta \eta^{1+\alpha} +\nabla \cdot \left(\eta\nabla c\right)=0 &\quad \text{in $\mathbb{R}^{3}_{T}$},  \\
\partial_{t}c+v\cdot \nabla c-\Delta c+c\eta=0 &\quad \text{in $\mathbb{R}^{3}_{T}$},\\
\partial_{t}v-\Delta v +\nabla p+\eta\nabla \phi=0 &\quad \text{in $\mathbb{R}^{3}_{T}$},\\
\nabla \cdot v=0 &\quad \text{in $\mathbb{R}^{3}_{T}$},\\
\eta(0,x)=\eta_{0}(x), \ \ c(0,x)=c_{0}(x), \ \ v(0,x)=v_{0}(x)
&\quad \text{in $\mathbb{R}^{3}$}.
\end{align}
\end{subequations}

Since the system \eqref{KS-PE} satisfies the assumptions in \cite[Theorem 1.10]{Chung 2}, it is straightforward that there exist H\"older continuous weak solutions when $\alpha>\frac{1}{8}$ and initial data are sufficiently regular. In this case, we can show that such solutions become unique. Our main result reads as follows:

\begin{theorem} \label{uniqueness theorem}
Let $\alpha>\frac{1}{8}$ and $(\eta_{0},c_{0},v_{0})$ satisfy
\begin{equation*}
\begin{split}
& \eta_{0} (1+|x|+|\ln \eta_{0}|)\in L^{1}(\mathbb{R}^{3}), \ \  \eta_{0}\in L^{\infty}(\mathbb{R}^{3}), \\
&  c_{0} \in L^{\infty}(\mathbb{R}^{3}) \cap
H^{1}(\mathbb{R}^{3})\cap W^{1,m}(\mathbb{R}^{3}), \ \ v_{0}\in
L^{2}(\mathbb{R}^{3})\cap W^{1,m}(\mathbb{R}^{3}) \quad \text{for
any $m<\infty$}.
\end{split}
\end{equation*}
Then, H\"older continuous weak solutions of the system \eqref{KS-PE} are unique.
\end{theorem}


We note that there are some known results regarding uniqueness of H\"older continuous weak solutions of Keller-Segel equations of the porous medium type (see \cite{Kim, Sugiyama}).
\begin{subequations} \label{PKS}
\begin{align}
\partial_{t}\eta-\Delta \eta^{1+\alpha} +\nabla \cdot \left(\eta\nabla c\right)=0 &\quad \text{in $\mathbb{R}^{3}_{T}$},  \\
\partial_{t}c-\Delta c+c-\eta=0 &\quad \text{in $\mathbb{R}^{3}_{T}$},\\
\eta(0,x)=\eta_{0}(x), \ \ c(0,x)=c_{0}(x) &\quad \text{in
$\mathbb{R}^{3}$}.
\end{align}
\end{subequations}

In principle, we apply  the duality argument and the vanishing
viscosity method used in  \cite{Sugiyama}. However, the nonlinearity
of the equation for $c$, presence of the fluid equations as well as
drift terms in the equations of $\eta$ and $c$, cause other kinds of
difficulties, which do not seem to allow the techniques in \cite{Sugiyama} directly applicable. Nevertheless, we prove the
uniqueness of H\"older continuous weak solutions of the system
(\ref{KS-PE}) via adapted methods of proofs in \cite{Sugiyama}
with estimates of the Green function of parabolic equations.

In the following we briefly explain how our methods proceed. Let $(\eta_{1}, c_{1}, v_{1}, p_{1})$ and $(\eta_{2}, c_{2}, v_{2},p_{2})$ be two H\"older continuous weak solutions of the system (\ref{KS-PE}) and let
\[
\eta=\eta_{1}-\eta_{2}, \quad c=c_{1}-c_{2}, \quad v=v_{1}-v_{2}, \quad p=p_{1}-p_{2}.
\]
We then see that $(\eta, c, v,p)$ solves
\begin{subequations} \label{PKS}
\begin{align}
\partial_{t}\eta+v_{1}\cdot \nabla \eta + v\cdot\nabla \eta_{2}-\Delta \left(\eta^{1+\alpha}_{1}-\eta^{1+\alpha}_{2}\right)
+\nabla \cdot \left(\eta\nabla c_{1}-\eta_{2}\nabla c\right)=0 &\quad \text{in $\mathbb{R}^{3}_{T}$},  \\
\partial_{t}c+v_{1}\cdot \nabla c+\eta_{2}c-\Delta c+c_{1}\eta + v\cdot \nabla c_{2}=0 &\quad \text{in $\mathbb{R}^{3}_{T}$},\\
\partial_{t}v-\Delta v+\nabla p+\eta\nabla \phi=0 &\quad \text{in $\mathbb{R}^{3}_{T}$},\\
\nabla\cdot v=0 &\quad \text{in $\mathbb{R}^{3}_{T}$}.
\end{align}
\end{subequations}

Next, we express $c$ and $v$ as integral forms involving $\eta$ by using the representation formula via the Green functions for parabolic equations with lower order terms and Stokes system in three dimensions. We then substitute them to the equation of $\eta$ to get an equation of the form $((\mathcal{D} \eta, \Phi)) = 0$ for any appropriate test function $\Phi$, where $\mathcal{D}$ is a differential operator involving $\eta_1, \eta_{2}, v_{1}, c_{1},\phi$ for $\eta$ and $((, \cdot, ))$ is the pairing in space and time. Let $\mathcal{D}^*$ be the adjoint operator for $\mathcal{D}$ defined at \eqref{def-D} in section 3. Then $((\mathcal{D} \eta, \Phi)) =((\eta, \mathcal{D}^{\ast} \Phi)) = 0$. If we can solve $\mathcal{D}^* \Phi= 0$ for $\Phi$, given any $\Phi_{0}\in X$ (which will specified below) it would follow that $\eta \equiv 0$, establishing uniqueness of solutions to the system (\ref{KS-PE}).

When we prove Theorem \ref{uniqueness theorem}, we formulate the dual problem in terms of a new function $\zeta$ defined in (\ref{omega equation}), and it contains the term $A^{1+\alpha}\Delta \zeta$, where $A^{1+\alpha}$ is defined in (\ref{eta quotient}). Since $A^{1+\alpha}$ is degenerate, we add $\delta \Delta$ to the equation of $\zeta$. After solving the equation of $\zeta^{\delta}$ for each $\delta>0$, we show that
\[
\lim_{\delta\rightarrow 0}\delta \int_{\mathbb{R}^{3}_{T}} f(\tau,x)\Delta \zeta^{\delta}(\tau,x)dxd\tau=0
\]
for all $f\in L^{2}(\mathbb{R}^{3}_{T})$.

\vspace{1ex}

One of main tools of proving Theorem \ref{uniqueness theorem}  is some point-wise estimates of the \emph{Green function of a parabolic equation with lower order terms of variable coefficients}, which seems to be of independent interest. More precisely, consider the equation of the form
\begin{equation} \label{equation of f}
\begin{split}
\partial_{t}f+a\cdot \nabla f+b f-\Delta f=F \quad &\text{in $\mathbb{R}^{3}_{T}$}, \\
f(0,x)=f_{0}(x) \quad &\text{in $\mathbb{R}^{3}$},
\end{split}
\end{equation}
where $a:\mathbb{R}^{3}_{T}\rightarrow \mathbb{R}^3$ and
$b:\mathbb{R}^{3}_{T}\rightarrow \mathbb{R}$ are given vector field
and scalar function, respectively. We then have the following:

\begin{theorem}\label{theorem c}
Let $a\in C^{\beta,\beta/2}\left(\mathbb{R}^{3}_{T}\right)$, $b \in C^{\beta,\beta/2}\left(\mathbb{R}^{3}_{T}\right)$ for some $0<\beta<1$, with $\nabla \cdot a=0$ and $b\ge 0$. Then, a solution of the system (\ref{equation of f}) can be written as
\[
f(t,x)=\int_{\mathbb{R}^{3}}\Gamma(t,x,0,y)f_{0}(y)dy+ \int^{t}_{0}\int_{\mathbb{R}^{3}}\Gamma(t,x,s,y)F(s,y)dyds,
\]
where $\Gamma(t,x,s,y)$ is the fundamental solution of a parabolic operator $\mathcal{L}$:
\[
\mathcal{L}=\partial_{t}-\Delta +a\cdot \nabla  +b.
\]
Moreover, $\Gamma(t,x,s,y)$ satisfies the following pointwise bounds:
\begin{equation}\label{gaussian011}
\abs{\Gamma(t,x,s,y)} \le C_0 (t-s)^{-\frac{3}{2}} \exp\left(-\frac{c\abs{x-y}^2}{t-s}\right),\quad (0\le s <t \le T),
\end{equation}
where $c=\frac14 -\epsilon$ for any $0<\epsilon <\frac14$ and $C_0 =C_0\bigl(T, \norm{a}_{L^\infty(\mathbb{R}^{3}_{T})}, \norm{b}_{L^\infty(\mathbb{R}^{3}_{T})}, \epsilon \bigr)$.
Moreover,
\begin{equation} \label{gaussian031}
\begin{split}
&\abs{\nabla_{x} \Gamma(t,x,s,y)} \le C_1 (t-s)^{-\frac{3+1}{2}} \exp\left(-\frac{c_{1}\abs{x-y}^2}{t-s}\right), \quad (0\le s <t \le T),\\
& \abs{\nabla_{x}^2 \Gamma(t,x,s,y)}+\abs{ \partial_t \Gamma(t,x,s,y)} \le C_2 (t-s)^{-\frac{3+2}{2}} \exp\left(-c_2\frac{\abs{x-y}^2}{t-s}\right), \quad (0\le s <t \le T)
\end{split}
\end{equation}
where $c_i>0$ is an absolute constant and $C_i =C_i  \bigl(T, \norm{a}_{C^{\beta, \beta/2}(\bR^3_T)}, \norm{b}_{C^{\beta,\beta/2}(\bR^3_T)}\bigr)$, $i=1,2$.
\end{theorem}

\begin{remark}
In fact, the assumption that $b \ge 0$ in Theorem~\ref{theorem c} is not so essential.
If we set $v(x,t)=e^{-\mu t}u(x,t)$,  then $v$ satisfies $\mathcal{L} v = e^{-\mu t}(\mathcal{L} u -\mu u)$, or equivalently,
\[
\tilde{\mathcal L} v :=\mathcal{L}v + \mu v = e^{-\mu t} \mathcal{L} u.
\]
The assumption that $b \in C^{\beta,\beta/2}\left(\mathbb{R}^{3}_{T}\right)$ particularly implies that $b$ is bounded, and thus we can make $b+\mu \ge 0$ by choosing $\mu$ large enough, for example, $\mu=\norm{b}_{L^\infty(\bR^3_T)} \le \norm{b}_{C^{\beta,\beta/2}(\bR^3_T)}$.
Note that the new operator $\tilde{\mathcal L}$ satisfies  the hypothesis of Theorem~\ref{theorem c}.
Therefore, we can apply Theorem~\ref{theorem c} to $v$ and transfer the results back to $u$.
\end{remark}

This paper is organized as follows. In section 2 we recall some useful notations and known results.  Section 3 is devoted to proving Theorem \ref{uniqueness theorem}. In section 4, we discuss the fundamental solutions of some parabolic equations and provide the proof of Theorem \ref{theorem c}.

\section{Preliminaries}

\subsection{Notations}
\noindent
\textbullet \ \ All generic constants will be denoted by $C$. We write $C=C(p_{1}, p_{2}, \cdots)$ to mean a constant that depends on $(p_{1}, p_{2}, \cdots)$. We follow the convention that such constants can vary from expression to expression and even between two occurrences within the same expression. $f^{\delta}$ denotes the dependence of a function on a parameter $\delta$.

\noindent
\textbullet \ \ $L^{p}$ and $W^{k,p}$ are the usual Lebesgue spaces and the Sobolev spaces. $L^{p}_{T}L^{q}$ denotes the Banach set of Bochner measurable functions $f$ from $(0,T)$ to $L^{q}(\mathbb{R}^{3})$ such that $\|f(t)\|_{L^{q}}\in L^{p}(0,T)$.

\noindent
\textbullet \ \ For $1<s,p<\infty$ and $0<T\leq \infty$, we define a function space
\[
\mathcal{P}^{s,p}_{T}(\mathbb{R}^{3})=\left\{f\in \mathcal{D}'((0,T)\times \mathbb{R}^{3}): \partial_{t}f\in L^{s}_{T}L^{p}(\mathbb{R}^{3}), \  f\in  L^{s}_{T}W^{2,p}(\mathbb{R}^{3})\right\}
\]
with the norm
\[
\|f\|_{\mathcal{P}^{s,p}_{T}}=\left\|\partial_{t}f\right\|_{L^{s}_{T}L^{p}} +\|f\|_{L^{s}_{T}W^{2,p}}.
\]
We also use the real interpolation space for initial data (\cite{Shatah}, Chapter 3)
\[
\mathcal{I}^{s,p}_{0}(\mathbb{R}^{3})=\left(L^{p}(\mathbb{R}^{3}), W^{2,p}(\mathbb{R}^{3})\right)_{1-\frac{1}{s},s}.
\]

\noindent
\textbullet \ \ We finally introduce the H\"older space $\mathcal{C}^{\beta,\gamma}(\mathbb{R}^{3}_{T})$, $0<\beta,\gamma<1$:
\[
\mathcal{C}^{\beta,\gamma}(\mathbb{R}^{3}_{T})=\left\{f\in \mathcal{C}(\mathbb{R}^{3}_{T}): \|f\|_{C^{\beta,\gamma}(\mathbb{R}^{3}_{T})}<\infty\right\},
\]
where
\[
\|f\|_{C^{\beta,\gamma}(\mathbb{R}^{3}_{T})}=\sup_{(t,x)\in \mathbb{R}^{3}_{T}}|f(t,x)| +\sup_{(t,x), (s,y) \in \mathbb{R}^{3}_{T}, \ x\ne y, s\ne t}\frac{|f(t,x)-f(s,y)|}{|t-s|^{\gamma}+|x-y|^{\beta}}.
\]
In particular, we set $\mathcal{C}^{\beta}(\mathbb{R}^{3}_{T})=\mathcal{C}^{\beta,\beta}(\mathbb{R}^{3}_{T})$.

\subsection{Parabolic equations}
To show uniqueness of solutions of the system (\ref{KS-PE}), we use the vanishing viscosity method. This requires to solve the equation of the form
\begin{equation}\label{auxiliary problem}
\begin{split}
\partial_{t}f-(\delta+V(t,x))\Delta f+\mu f=g(t,x) &\quad \text{in $\mathbb{R}^{3}_{T}$}\\
f(0,x)=f_{0}(x) &\quad \text{in $\mathbb{R}^{3}$}.
\end{split}
\end{equation}

\begin{lemma}\cite[Lemma 3.4]{Sugiyama} \label{auxiliary lemma}
Let $1<s,p<\infty$ and $0<T<\infty$. Suppose $V$ satisfies the following conditions:
\[
V\in L^{\infty}_{T}\left(L^{1}\cap L^{\infty}\right) \bigcap C^{\beta}(\mathbb{R}^{3}_{T}), \quad V(t,x)\ge 0
\]
for some $0<\beta<1$. Then for every $g\in L^{s}_{T}L^{p}$ and $f_{0}\in \mathcal{I}^{s,p}_{0}$, there exists a unique solution $f$ of (\ref{auxiliary problem}) on $(0,T)$. Moreover, there exists
\[
\mu_{1}=\mu_{1}(\beta, \|V\|_{C^{\beta}}, \delta,s,p,T)
\]
such that
\[
\left\|\partial_{t}f\right\|_{L^{s}_{T}L^{p}}+\delta\left\|\nabla^{2}f\right\|_{L^{s}_{T}L^{p}}+\mu\left\|f\right\|_{L^{s}_{T}L^{p}} \leq C \left(\left\|g\right\|_{L^{s}_{T}L^{p}}+\left\|f_{0}\right\|_{\mathcal{I}^{s,p}_{0}}\right)
\]
for all $\mu\ge \mu_{1}$, where $C=C(\beta, \|V\|_{C^{\beta}}, \delta,s,p,T)$ is a constant independent of $\mu\ge \mu_{1}$.
\end{lemma}

We next consider the Stokes system
\begin{subequations} \label{Stokes system}
\begin{align}
\partial_{t}v-\Delta v +\nabla p=F &\quad \text{in $\mathbb{R}^{3}_{T}$},\\
\nabla \cdot v=0 &\quad \text{in $\mathbb{R}^{3}_{T}$},\\
v(0,x)=v_{0}(x) &\quad \text{in $\mathbb{R}^{3}$}.
\end{align}
\end{subequations}
Let $\mathbb{P}$ be the Leray projection operator. Then, we can write $v$ as
\[
v(t,x)=e^{t\Delta}v_{0}+\int^{t}_{0}e^{(t-s)\Delta}\mathbb{P}F(s)ds,
\]
where $e^{t\Delta}$ is the heat kernel. We note that $S(t):=e^{t\Delta}\mathbb{P}$ has the following $L^{p}$ estimates: for $1\leq p\leq r\leq \infty$
\begin{equation} \label{decay rate of u}
\begin{split}
&\left\|S(t)f\right\|_{L^{r}}\leq C(p,r) t^{-\frac{3}{2}\left(\frac{1}{p}-\frac{1}{r}\right)}\|f\|_{L^{p}}, \\
&\left\|\nabla S(t) f\right\|_{L^{r}}\leq C(p,r) t^{-\frac{3}{2}\left(\frac{1}{p}-\frac{1}{r}\right)-\frac{1}{2}}\|f\|_{L^{p}}.
\end{split}
\end{equation}

Without the operator $\mathbb{P}$, we have the bounds (\ref{decay rate of u}) from the heat kernel $e^{t\Delta}$ \cite{Sugiyama}. Since the operator $\mathbb{P}$ is a singular integral operator of degree 0, it  is a bounded operator such that 
\[
\left\|\mathbb{P}f\right\|_{L^{p}}\leq C(p)\|f\|_{L^{p}}
\]
and thus we have (\ref{decay rate of u}) as well.

We also consider the equation:
\begin{equation} \label{equation of f 1}
\begin{split}
\partial_{t}f+a\cdot \nabla f+b f-\Delta f=0 \quad &\text{in $\mathbb{R}^{3}_{T}$}, \\
f(0,x)=f_{0}(x) \quad &\text{in $\mathbb{R}^{3}$}.
\end{split}
\end{equation}
Under the assumptions on $a$ and $b$ in Theorem \ref{theorem c}, (\ref{gaussian011}) and (\ref{gaussian031}) imply the following estimates: for $1\leq p\leq r\leq \infty$
\begin{equation} \label{decay rate of c}
\begin{split}
&\left\|\int_{\mathbb{R}^{3}}\Gamma(t,x,0,y) f(y)dy\right\|_{L^{r}}\leq C(p,r,a,b) t^{-\frac{3}{2}\left(\frac{1}{p}-\frac{1}{r}\right)}\|f\|_{L^{p}},\\
&\left\|\int_{\mathbb{R}^{3}}\nabla_{x}\Gamma(t,x,0,y) f(y)dy\right\|_{L^{r}}\leq C(p,r,a,b) t^{-\frac{3}{2}\left(\frac{1}{p}-\frac{1}{r}\right)-\frac{1}{2}}\|f\|_{L^{p}}.
\end{split}
\end{equation}

\section{Proof of Theorem \ref{uniqueness theorem}}

In order to show the uniqueness of H\"older continuous weak solutions, let $(\eta_{1}, c_{1}, v_{1}, p_{1})$ and $(\eta_{2}, c_{2}, v_{2}, p_{2})$ be two solutions of the system (\ref{KS-PE}) and let
\[
\eta=\eta_{1}-\eta_{2}, \quad c=c_{1}-c_{2}, \quad v=v_{1}-v_{2}, \quad p=p_{1}-p_{2}.
\]
Then, $(\eta, c, v, p)$ satisfies the following equations in $\mathbb{R}^{3}_{T}$:
\begin{subequations}\label{difference equations dd}
\begin{align}
\partial_{t}\eta+v_{1}\cdot \nabla \eta + v\cdot\nabla \eta_{2}-\Delta \left(\eta^{1+\alpha}_{1}-\eta^{1+\alpha}_{2}\right)+\nabla \cdot \left(\eta\nabla c_{1}-\eta_{2}\nabla c\right)=0 &\quad \text{in $\mathbb{R}^{3}_{T}$},\label{difference eta}\\
\partial_{t}c+v_{1}\cdot \nabla c+\eta_{2}c-\Delta c+c_{1}\eta + v\cdot \nabla c_{2}=0 &\quad \text{in $\mathbb{R}^{3}_{T}$},\label{difference c}\\
\partial_{t}v-\Delta v+\nabla p+\eta\nabla \phi=0 &\quad \text{in $\mathbb{R}^{3}_{T}$},\\
\nabla\cdot v=0 &\quad \text{in $\mathbb{R}^{3}_{T}$},\\
\eta(0,x)=c(0,x)=v(0,x)=0 &\quad \text{in $\mathbb{R}^{3}$}.
\end{align}
\end{subequations}

We first express $v$ in terms of $\eta$:
\eqn \label{integral of u}
v(t,x)=-\int^{t}_{0}\int_{\mathbb{R}^{3}}S(t-s,x-y)\left(\eta \nabla \phi\right)(s,y)dy ds.
\een
By Theorem \ref{theorem c}, we also express $c$ as
\begin{equation}\label{integral of c}
\begin{split}
c(t,x)=-\int^{t}_{0}\int_{\mathbb{R}^{3}}\Gamma(t,x,s,y)\left(c_{1}\eta+v\cdot \nabla c_{2}\right)(s,y)dyds,
\end{split}
\end{equation}
where $\Gamma$ is the fundamental solution in Theorem \ref{theorem c} with $a$ and $b$ replaced by $v_{1}$ and $\eta_{2}$.  So, we can reformulate the dual system of (\ref{difference equations
dd}) by the dual problem of $\eta$. Once uniqueness of $\eta$ is
proved, uniqueness of $c$ and $v$ then comes automatically.

To write the dual equation of $\eta$, let \eqn \label{eta quotient}
A^{\alpha+1}(t):=A^{\alpha+1}\left(\eta_{1}(t,x),
\eta_{2}(t,x)\right)=\frac{\eta^{\alpha+1}_{1}(t,x)-\eta^{\alpha+1}_{2}(t,x)}{\eta_{1}(t,x)-\eta_{2}(t,x)}.
\een We multiply (\ref{difference eta}) by $\Phi\in
C^{\infty}_{c}(\mathbb{R}^{3}_{T})$. Then, via the integration by
parts, we have
\begin{equation*}
\begin{split}
& \int_{\mathbb{R}^{3}}\eta(t,x) \Phi(t,x)dx + \int^{t}_{0}\int_{\mathbb{R}^{3}} \eta(\tau,x) \delta \Delta\Phi(\tau,x)dx d\tau\\
&= \int^{t}_{0}\int_{\mathbb{R}^{3}}\eta(\tau,x)\partial_{\tau}\Phi(\tau,x) dxd\tau +\int^{t}_{0}\int_{\mathbb{R}^{3}}\eta(\tau,x) \left(\delta \Delta\Phi(\tau)+A^{\alpha+1}(\tau)\Delta\Phi(\tau,x)\right)dxd\tau \\
& +\int^{t}_{0}\int_{\mathbb{R}^{3}} \eta(\tau,x) (v_{1}\cdot \nabla \Phi)(\tau,x)dx d\tau  +\int^{t}_{0}\int_{\mathbb{R}^{3}}\eta(\tau,x)\left(\nabla c_{1}\cdot \nabla \Phi\right)(\tau,x)dx d\tau\\
& +\int^{t}_{0}\int_{\mathbb{R}^{3}} \eta_{2}(\tau,x)(v\cdot \nabla \Phi)(\tau,x)dx d\tau -\int^{t}_{0}\int_{\mathbb{R}^{3}} (\eta_{2}\nabla c)(\tau,x) \nabla \Phi(\tau,x)dx d\tau\\
& =\text{$N_{1}$+$N_{2}$+$N_{3}$+$N_{4}$+$N_{5}$+$N_{6}$},
\end{split}
\end{equation*}
where we add $\delta \Delta$ because $A^{\alpha+1}(\tau)$ is degenerate. We note that $N_{1}$, $N_{2}$, $N_{3}$ and $N_{4}$ are already given in the dual form. By (\ref{integral of u}),
\begin{equation*}
\begin{split}
\text{$N_{5}$}=-\int^{t}_{0}\int_{\mathbb{R}^{3}}\eta(\tau,x) \nabla \phi(x)\cdot \left[\int^{t}_{\tau}\int_{\mathbb{R}^{3}}S(\widehat{\tau}-\tau, x-y) \left(\eta_{2}(\widehat{\tau},y)\nabla\Phi(\widehat{\tau},y)\right)dyd\widehat{\tau}\right]dx d\tau.
\end{split}
\end{equation*}
Also, by (\ref{integral of c}) we rewrite $\text{$N_{6}$}$ as
\begin{equation*}
\begin{split}
\text{$N_{6}$}& =\int^{t}_{0}\int_{\mathbb{R}^{3}} \eta(\tau,x) c_{1}(\tau,x) \left[\int^{t}_{\tau} \int_{\mathbb{R}^{3}}\nabla_{y} \Gamma\left(\widehat{\tau},y,\tau,x\right) \cdot\left(\eta_{2}(\widehat{\tau},y)\nabla\Phi(\widehat{\tau},y)\right)dyd\widehat{\tau}\right]dxd\tau\\
& +\int^{t}_{0}\int_{\mathbb{R}^{3}} v(\tau,x) \cdot \nabla c_{2}(\tau,x) \left[\int^{t}_{\tau} \int_{\mathbb{R}^{3}}\nabla_{y} \Gamma\left(\widehat{\tau},y,\tau,x\right) \cdot\left(\eta_{2}(\widehat{\tau},y)\nabla\Phi(\widehat{\tau},y)\right)dyd\widehat{\tau}\right] dxd\tau\\
&=\text{$N_{6_{1}}$}+\text{$N_{6_{2}}$}.
\end{split}
\end{equation*}
Using (\ref{integral of u}), we proceed further to rewrite $\text{$N_{6_{2}}$}$ as
\begin{equation*}
\begin{split}
\text{$N_{6_{2}}$}=-\int^{t}_{0}\int_{\mathbb{R}^{3}}&\eta(\tau,x) \nabla \phi (x)  \cdot\Big\{\int^{t}_{\tau}\int_{\mathbb{R}^{3}}S({\widehat{\tau}-\tau},x-y) \nabla c_{2}(\widehat{\tau},y)\\
& \times \Big[\int^{t}_{\widehat{\tau}}\int_{\mathbb{R}^{3}} \nabla_{z} \Gamma(\rho,z, \widehat{\tau},y) \cdot\left(\eta_{2}(\rho,z)\nabla \Phi(\rho,z)\right)dzd\rho\Big]dyd\widehat{\tau}\Big\}dx d\tau.
\end{split}
\end{equation*}

Collecting all terms together, we obtain
\eqn \label{a priori each delta}
\int_{\mathbb{R}^{3}}\eta(t,x) \Phi(t,x)dx + \delta\int_{\mathbb{R}^{3}_{t}} \eta(\tau,x) \Delta\Phi(\tau,x)dx d\tau=\int_{\mathbb{R}^{3}_{t}}\eta(\tau,x)\mathcal{D}^{\ast}(\tau,x)dx d\tau,
\een
where
\begin{equation}\label{def-D}
\begin{split}
\mathcal{D}^{\ast}(\tau,x)&= \partial_{\tau}\Phi +\delta \Delta \Phi +A^{\alpha+1}\Delta\Phi +v_{1}\cdot \nabla \Phi + \nabla c_{1}\cdot \nabla \Phi \\
&- \nabla \phi(x)\cdot \int^{t}_{\tau}\int_{\mathbb{R}^{3}}S(\widehat{\tau}-\tau, x-y) \eta_{2}(\widehat{\tau},y)\nabla\Phi(\widehat{\tau},y) dyd\widehat{\tau}\\
&+c_{1}(\tau,x) \int^{t}_{\tau} \int_{\mathbb{R}^{3}}\nabla_{y} \Gamma\left(\widehat{\tau},y,\tau,x\right) \cdot\left(\eta_{2}(\widehat{\tau},y)\nabla\Phi(\widehat{\tau},y)\right)dyd\widehat{\tau} \\
&-\nabla \phi \cdot\int^{t}_{\tau}\int_{\mathbb{R}^{3}} S({\widehat{\tau}-\tau},x-y) \nabla c_{2}(\widehat{\tau},y)\\
& \times \Big[\int^{t}_{\widehat{\tau}}\int_{\mathbb{R}^{3}} \nabla_{z} \Gamma(\rho, z,\widehat{\tau},y) \left(\eta_{2}(\rho,z)\nabla \Phi(\rho,z)\right)dzd\rho \Big]dyd\widehat{\tau}.
\end{split}
\end{equation}

At this point, we could finish the proof of Theorem  \ref{uniqueness theorem}  if we could show the followings.
\begin{enumerate}[]
\item \textbullet \ \ Find a solution $\Phi^{\delta}$ solving $\mathcal{D}^{\ast}(\tau,x)=0$ for a.e. $(\tau,x)\in \mathbb{R}^{3}_{t}$ for each $\delta>0$.
\item \textbullet \ \ $\displaystyle \lim_{\delta\rightarrow 0} \delta \int^{t}_{0}\int_{\mathbb{R}^{3}} \eta(\tau,x)\Delta\Phi^{\delta}(\tau,x)dx d\tau=0$.
\end{enumerate}
Let $\displaystyle \lim_{\delta\rightarrow 0}\Phi^{\delta}=\Phi$ in $\mathcal{P}^{s,p}_{T}$. Then, (\ref{a priori each delta}) implies that
\eqn \label{almost last}
\int_{\mathbb{R}^{3}}\eta(t,x) \Phi(t,x)dx=0.
\een
But, in order to say $\eta \equiv 0$ from  (\ref{almost last}), we should remove the time dependency in $\Phi(t,x)$.  To do this, let $\zeta^{\delta}(t-\tau,x)=\Phi^{\delta}(\tau,x)$. Then, $\Phi^{\delta}(t,x)=\zeta^{\delta}(0,x)$. Let
\[
\theta=t-\tau, \quad \widehat{\theta}=t-\widehat{\tau}, \quad \rho=t-\sigma.
\]
Then, $\mathcal{D}^{\ast}(\tau,x)=0$ is equivalent to solve the following equation for each $\delta>0$
\begin{equation}\label{omega equation}
\begin{split}
\partial_{\theta}\zeta^{\delta} &-\left(\delta \Delta +A^{\alpha+1}(t-\theta)\right)\zeta^{\delta} -v_{1}(t-\theta)\cdot \nabla \zeta^{\delta} - \nabla c_{1}(t-\theta)\cdot \nabla \zeta^{\delta} \\
&+\nabla \phi\int^{\theta}_{0}\int_{\mathbb{R}^{3}} S(\theta-\widehat{\theta}, x-y) \eta_{2}(t-\widehat{\theta},y)\nabla\zeta^{\delta}(\widehat{\theta},y)dyd\widehat{\theta}  \\
&-c_{1}(t-\tau)\int^{\theta}_{0}\int_{\mathbb{R}^{3}}\nabla_{y} \Gamma(t-\widehat{\theta}, y, t-\theta,x) \cdot\left(\eta_{2}(t-\widehat{\theta},y)\nabla\zeta^{\delta}(\widehat{\theta},y)\right)dyd\widehat{\theta}\\
&+\nabla \phi \cdot \int^{\theta}_{0}\int_{\mathbb{R}^{3}} S(\theta-\widehat{\theta},x-y)\nabla c_{2}(t-\widehat{\theta},y) \mathcal{T}(\widehat{\theta},t,y,\zeta^{\delta}) dyd\widehat{\theta}=0,
\end{split}
\end{equation}
where
\eqn \label{complicated form}
\mathcal{T}(\widehat{\theta},t,y,\zeta^{\delta})=\int^{\widehat{\theta}}_{0}\int_{\mathbb{R}^{3}} \nabla_{z} \Gamma(t-\sigma,z, t-\widehat{\theta},y) \eta_{2}(t-\sigma,z)\nabla \zeta^{\delta}(\sigma,z)dz d\sigma.
\een

\begin{lemma}\label{solution omega}
Let $\alpha>\frac{1}{8}$, $2<s,p<\infty$ and $0<t<T$. For any $\zeta_{0}\in \mathcal{I}^{s,p}_{0}$, there exists a unique solution $\zeta^{\delta}\in \mathcal{P}^{s,p}_{T}$ of (\ref{omega equation}) for each $\delta>0$. Moreover,  for a.e. $0<t<T$ and for all $\chi \in L^{2}_{t}L^{2}$,
\[
\lim_{\delta\rightarrow 0}\delta \int^{t}_{0} \int_{\mathbb{R}^{3}}\chi(\tau,x) \Delta \zeta^{\delta}(\tau,x)dxd\tau=0.
\]
\end{lemma}

If Lemma \ref{solution omega} can be proved, the proof of Theorem  \ref{uniqueness theorem} can be finalized. More precisely, by the dual problem (\ref{omega equation}), $\mathcal{D}^{\ast}(\tau,x)=0$ for a.e. $(\tau,x)\in \mathbb{R}^{3}_{T}$. So,
\[
\int_{\mathbb{R}^{3}}\eta(t,x)\zeta^{\delta}_{0}(x)dx + \delta \int_{\mathbb{R}^{3}_{t}} \eta(\tau,x) \Delta\Phi^{\delta}(\tau,x)dxd\tau=0.
\]
Moreover, the second part of Lemma \ref{solution omega}, with $\displaystyle \lim_{\delta\rightarrow 0}\zeta^{\delta}_{0}=\zeta_{0}$, implies that
\[
\int_{\mathbb{R}^{3}}\eta(t,x)\zeta_{0}(x)dx=0
\]
for a.e. $t\in (0,T)$. Since $\zeta_{0}\in \mathcal{I}^{s,p}_{0}$ is arbitrary, we conclude that $\eta \equiv 0$.

\subsection{Proof of Lemma \ref{solution omega}}
To solve the equation (\ref{omega equation}), we first introduce the exponential factor:
\[
\zeta^{\delta}(\theta,x)=e^{\mu \theta}\psi^{\delta}(\theta,x),
\]
where $\mu$ is a constant to be determined when we apply Lemma \ref{auxiliary lemma} later. Then, $\psi^{\delta}$ satisfies
\begin{equation*}
\begin{split}
\partial_{\theta}\psi^{\delta} &-\left(\delta \Delta +A^{\alpha+1}(t-\theta)\right)\psi^{\delta} +\mu\psi^{\delta} - v_{1}(t-\theta)\cdot \nabla \psi^{\delta} - \nabla c_{1}(t-\theta)\cdot \nabla\psi^{\delta} \\
&+\nabla \phi\int^{\theta}_{0}\int_{\mathbb{R}^{3}} S(\theta-\widehat{\theta}, x-y) \eta_{2}(t-\widehat{\theta},y)e^{\mu(\widehat{\theta}-\theta)}\nabla\psi^{\delta}(\widehat{\theta},y)dyd\widehat{\theta}  \\
&-c_{1}(t-\tau)\int^{\theta}_{0}\int_{\mathbb{R}^{3}}\nabla_{y} \Gamma(t-\widehat{\theta},y,t-\theta,x) \cdot\left(\eta_{2}(t-\widehat{\theta},y) e^{\mu(\widehat{\theta}-\theta)} \nabla\psi^{\delta}(\widehat{\theta},y)\right)dyd\widehat{\theta}\\
&+\nabla \phi\cdot \int^{\theta}_{0}\int_{\mathbb{R}^{3}} S(\theta-\widehat{\theta},x-y)\nabla c_{2}(t-\widehat{\theta},y) \mathcal{T}(\widehat{\theta},t,y,\psi^{\delta})dyd\widehat{\theta}=0,
\end{split}
\end{equation*}
where $\mathcal{T}(\widehat{\theta},t,y,\psi^{\delta})$ is defined by (\ref{complicated form}) with $\zeta^{\delta}\mapsto e^{\mu(\sigma-\theta)}\psi^{\delta}$. Since the exponential factors $e^{\mu(\widehat{\theta}-\theta)}$ and $e^{\mu(\sigma-\theta)}$ are less than 1, they do not affect the arguments below. Hence, we consider the following equation, with the same notation $\psi^{\delta}$,
\begin{equation}\label{psi equation delta}
\begin{split}
\partial_{\theta}\psi^{\delta} &-\left(\delta \Delta +A^{\alpha+1}(t-\theta)\right)\psi^{\delta} +\mu\psi -v_{1}(t-\theta)\cdot \nabla \psi^{\delta} - \nabla c_{1}(t-\theta)\cdot \nabla\psi^{\delta} \\
&+\nabla \phi\int^{\theta}_{0}\int_{\mathbb{R}^{3}} S(\theta-\widehat{\theta}, x-y) \eta_{2}(t-\widehat{\theta},y) \nabla\psi^{\delta}(\widehat{\theta},y)dyd\widehat{\theta}  \\
&-c_{1}(t-\tau)\int^{\theta}_{0}\int_{\mathbb{R}^{3}}\nabla_{y} \Gamma(t-\widehat{\theta},y, t-\theta,x) \cdot\left(\eta_{2}(t-\widehat{\theta},y)  \nabla\psi^{\delta}(\widehat{\theta},y)\right)dyd\widehat{\theta}\\
&+\nabla \phi \cdot \int^{\theta}_{0}\int_{\mathbb{R}^{3}} S(\theta-\widehat{\theta},x-y)\nabla c_{2}(t-\widehat{\theta},y)\mathcal{T}(\widehat{\theta},t,y,\psi^{\delta})dyd\widehat{\theta}=0.
\end{split}
\end{equation}

So, instead of proving Lemma \ref{solution omega}, we prove the following equivalent lemma.

\begin{lemma}\label{main lemma}
Let $\alpha>\frac{1}{8}$, $2<s,p<\infty$ and $0<t<T$. For any $\psi_{0}\in \mathcal{I}^{s,p}_{0}$, there exists a unique solution $\psi^{\delta}\in \mathcal{P}^{s,p}_{T}$ of (\ref{psi equation delta}) for each $\delta>0$. Moreover, for a.e. $0<t<T$ and for all $\chi \in L^{2}_{t}L^{2}$,
\[
\lim_{\delta\rightarrow 0}\delta \int^{t}_{0} \int_{\mathbb{R}^{3}}\chi(\tau,x) \Delta \psi^{\delta}(\tau,x)dxd\tau=0.
\]
\end{lemma}

The proof of Lemma \ref{main lemma} consists of two parts.

\subsubsection{Unique solvability of (\ref{psi equation delta})}
We construct a unique solution to the equation (\ref{psi equation delta}) via the iteration argument. Let $\psi_{1}=e^{\tau\Delta}\psi_{0}$ and and for $k\ge 2$,
\eqn \label{psi k equation}
\partial_{\theta}\psi^{\delta}_{k} -\left(\delta+A^{\alpha+1}(t-\theta)\right)\Delta\psi^{\delta}_{k} +\mu\psi^{\delta}_{k}=G_{k-1},
\een
where we choose $G_{1}$ by putting $\psi_{1}=e^{\tau\Delta} \psi_{0}$ in place of $\psi^{\delta}_{k-1}$ below and 
\begin{equation*}
\begin{split}
G_{k-1}&=v_{1}(t-\theta)\cdot \nabla \psi^{\delta}_{k-1} + \nabla c_{1}(t-\theta)\cdot \nabla \psi^{\delta}_{k-1} \\
&-\nabla \phi\int^{\theta}_{0}\int_{\mathbb{R}^{3}} S(\theta-\widehat{\theta}, x-y) \eta_{2}(t-\widehat{\theta},y) \nabla\psi^{\delta}_{k-1}(\widehat{\theta},y)dyd\widehat{\theta}  \\
&+c_{1}(t-\tau)\int^{\theta}_{0}\int_{\mathbb{R}^{3}}\nabla_{y} \Gamma(t-\widehat{\theta},y, t-\theta,x) \cdot\left(\eta_{2}(t-\widehat{\theta},y)  \nabla\psi^{\delta}_{k-1}(\widehat{\theta},y)\right)dyd\widehat{\theta}\\
&-\nabla \phi\cdot \int^{\theta}_{0}\int_{\mathbb{R}^{3}} S(\theta-\widehat{\theta},x-y)\nabla c_{2}(t-\widehat{\theta},y)\mathcal{T}(\widehat{\theta},t,y,\psi^{\delta}_{k-1})dyd\widehat{\theta}\\
&=\text{F}_{1}+\text{F}_{2}+\text{F}_{3}+\text{F}_{4}+\text{F}_{5}.
\end{split}
\end{equation*}

We now estimate $G_{k-1}$ in $L^{s}_{T}L^{p}$. First,
\[
\|\text{F}_{1}\|_{L^{s}_{T}L^{p}}+ \|\text{F}_{2}\|_{L^{s}_{T}L^{p}}\leq \left(\|v_{1}\|_{L^{\infty}_{T}L^{\infty}} + \left\|\nabla c_{1}\right\|_{L^{\infty}_{T}L^{\infty}}\right)\left\|\nabla \psi^{\delta}_{k-1}\right\|_{L^{s}_{T}L^{p}}.
\]

Using (\ref{decay rate of u}), we also have
\[
\|\text{F}_{3}\|_{L^{s}_{T}L^{p}}  \leq C(s,p)(1+T)^{2}\left\|\nabla \phi\right\|_{L^{\infty}}\left\|\eta_{2}\right\|_{L^{\infty}_{T}L^{\infty}} \left\|\nabla \psi^{\delta}_{k-1}\right\|_{L^{s}_{T}L^{p}}.
\]

By (\ref{decay rate of c}),
\[
 \|\text{F}_{4}\|_{L^{p}}\leq C_{1}\|c_{1}\|_{L^{\infty}_{T}L^{\infty}}\|\eta_{1}\|_{L^{\infty}_{T}L^{\infty}}\int^{\theta}_{0}\frac{1}{\sqrt{\theta-\widehat{\theta}}}\left\|\nabla \psi^{\delta}_{k-1}(\widehat{\theta})\right\|_{L^{p}} d\widehat{\theta}
\]
and hence we have, with $s>2$, 
\[
\|\text{F}_{4}\|_{L^{s}_{T}L^{p}} \leq C_{1}(1+T)^{2}\|c_{1}\|_{L^{\infty}_{T}L^{\infty}}\|\eta_{2}\|_{L^{\infty}_{T}L^{\infty}} \left\|\nabla \psi^{\delta}_{k-1}\right\|_{L^{s}_{T}L^{p}},
\]
where $C_{1}$ is the constant defined in Theorem \ref{theorem c}.

Similarly,
\[
\|\text{F}_{5}\|_{L^{s}_{T}L^{p}} \leq C_{1}(1+T)^{4}\|\nabla \phi\|_{L^{\infty}} \|\nabla c_{2}\|_{L^{\infty}_{T}L^{\infty}}\|\eta_{2}\|_{L^{\infty}_{T}L^{\infty}} \left\|\nabla \psi^{\delta}_{k-1}\right\|_{L^{s}_{T}L^{p}}.
\]

By Lemma \ref{auxiliary lemma},
\begin{equation*}
\begin{split}
\left\|\partial_{t}\psi^{\delta}_{k}\right\|_{L^{s}_{T}L^{p}}+\delta\left\|\nabla^{2}\psi^{\delta}_{k}\right\|_{L^{s}_{T}L^{p}}+\mu\left\|\psi^{\delta}_{k}\right\|_{L^{s}_{T}L^{p}}  \leq C_{\ast}\left\|\psi_{0}\right\|_{\mathcal{I}^{s,p}_{0}} +C_{\ast\ast}\left\|\nabla \psi^{\delta}_{k-1}\right\|_{L^{s}_{T}L^{p}},
\end{split}
\end{equation*}
where
\begin{equation*}
\begin{split}
& \mu\ge \mu_{1}(\beta, \|\eta_{1}\|_{C^{\beta}}, \|\eta_{2}\|_{C^{\beta}}, \delta,s,p,T),\quad  C_{\ast}=C(\beta, \|\eta_{1}\|_{C^{\beta}(\mathbb{R}^{3}_{T})}, \|\eta_{2}\|_{C^{\beta}(\mathbb{R}^{3}_{T})}, \delta,s,p,T),\\
& C_{\ast\ast}=C\left(C_{\ast}, \left\|\nabla \phi\right\|_{L^{\infty}}, \|v_{1}\|_{L^{\infty}_{T}L^{\infty}}, \|c_{1}\|_{L^{\infty}_{T}W^{1,\infty}}, \|\nabla c_{2}\|_{L^{\infty}_{T}L^{\infty}}, \left\|\eta_{2}\right\|_{L^{\infty}_{T}L^{\infty}}, \left\|\eta_{1}\right\|_{L^{\infty}_{T}L^{\infty}}, T\right).
\end{split}
\end{equation*}
Using the following interpolation
\[
\left\|\nabla \psi^{\delta}_{k-1}\right\|_{L^{p}}\leq \epsilon \left\|\nabla^{2}\psi^{\delta}_{k-1}\right\|_{L^{p}}+\frac{4}{\epsilon}\|\psi^{\delta}_{k-1}\|_{L^{p}}, \quad \epsilon=\frac{\delta}{2C_{\ast\ast}},
\]
we have
\begin{equation*}
\begin{split}
\left\|\partial_{t}\psi^{\delta}_{k}\right\|_{L^{s}_{T}L^{p}}+\delta\left\|\nabla^{2}\psi^{\delta}_{k}\right\|_{L^{s}_{T}L^{p}}+\mu\left\|\psi^{\delta}_{k}\right\|_{L^{s}_{T}L^{p}} \leq C_{\ast}\left\|\psi_{0}\right\|+\frac{\delta}{2}\left\|\nabla^{2}\psi^{\delta}_{k-1}\right\|_{L^{s}_{T}L^{p}} + \frac{8C^{2}_{\ast\ast}}{\delta}\left\|\psi^{\delta}_{k-1}\right\|_{L^{s}_{T}L^{p}}.
\end{split}
\end{equation*}

Let
\[
\mu\ge \max\left\{\mu_{1}, \frac{8C^{2}_{\ast\ast}}{\delta}\right\}.
\]
Then, we finally obtain
\begin{equation}\label{iterated a priori bound}
\begin{split}
&\left\|\partial_{t}\psi^{\delta}_{k}\right\|_{L^{s}_{T}L^{p}}+\delta\left\|\nabla^{2}\psi^{\delta}_{k}\right\|_{L^{s}_{T}L^{p}}+\mu\left\|\psi^{\delta}_{k}\right\|_{L^{s}_{t}L^{p}} \\
& \leq C_{\ast}\left\|\psi_{0}\right\|_{\mathcal{I}^{s,p}_{0}} +\frac{1}{2}\left(\left\|\partial_{t}\psi^{\delta}_{k-1}\right\|_{L^{s}_{T}L^{p}}+\delta\left\|\nabla^{2}\psi^{\delta}_{k-1}\right\|_{L^{s}_{T}L^{p}}+\mu\left\|\psi^{\delta}_{k-1}\right\|_{L^{s}_{T}L^{p}}\right).
\end{split}
\end{equation}

Let
\[
M(T)=\left\|\psi_{0}\right\|_{\mathcal{I}^{s,p}_{0}}\sup_{0\leq t\leq T}C_{\ast}.
\]
Then, (\ref{iterated a priori bound}) implies that the sequence $\{\psi_{k}\}$ is uniformly in $\mathcal{P}^{s,p}_{T}$, bounded above by $2M(T)$. Moreover, using the same argument, we can derive that
\begin{equation}\label{iterated a priori bound difference}
\begin{split}
&\left\|\partial_{t}\left(\psi^{\delta}_{k+1}-\psi^{\delta}_{k}\right)\right\|_{L^{s}_{T}L^{p}}+\delta\left\|\nabla^{2}\left(\psi^{\delta}_{k+1}-\psi^{\delta}_{k}\right)\right\|_{L^{s}_{T}L^{p}}+\mu\left\|\psi^{\delta}_{k+1}-\psi^{\delta}_{k}\right\|_{L^{s}_{T}L^{p}} \\
& \leq \frac{1}{2}\left(\left\|\partial_{t}\left(\psi^{\delta}_{k}-\psi^{\delta}_{k-1}\right)\right\|_{L^{s}_{T}L^{p}}+\delta\left\|\nabla^{2}\left(\psi^{\delta}_{k}-\psi^{\delta}_{k-1}\right)\right\|_{L^{s}_{T}L^{p}}+\mu\left\|\psi^{\delta}_{k}-\psi^{\delta}_{k-1}\right\|_{L^{s}_{T}L^{p}}\right).
\end{split}
\end{equation}
Hence $\left\{\psi^{\delta}_{k}\right\}$ is a Cauchy sequence in $\mathcal{P}^{s,p}_{T}$. Therefore, there exists a unique $\psi^{\delta}\in \mathcal{P}^{s,p}_{T}$ such that
\[
\lim_{k\rightarrow \infty}\left\|\psi^{\delta}_{k}-\psi^{\delta}\right\|_{\mathcal{P}^{s,p}_{T}}=0.
\]
By taking $k\rightarrow \infty$ to (\ref{psi k equation}), we obtain a unique solution $\psi^{\delta}$of (\ref{psi equation delta}) for each $\delta>0$.

\subsubsection{Vanishing viscosity of (\ref{psi equation delta})}
We now use the vanishing viscosity limit to (\ref{psi equation delta}). We multiply the equation (\ref{psi equation delta}) by $-\Delta \psi^{\delta}$ and integrate it over $\mathbb{R}^{3}$. Then, for $0<\theta<t$
\begin{equation}\label{a priori vanishing viscosity}
\begin{split}
&\frac{1}{2}\frac{d}{d\theta}\left\|\nabla \psi^{\delta}(\theta)\right\|^{2}_{L^{2}}+\mu \left\|\nabla \psi^{\delta}\right\|^{2}_{L^{2}}+\int\left(\delta+ A^{\alpha+1}(t-\theta)\right)\left|\Delta \psi^{\delta}\right|^{2}dx\\
&=\int_{\mathbb{R}^{3}} \left(v_{1}(t-\theta)\cdot \nabla \psi^{\delta}\right)\Delta \psi^{\delta} dx+\int \left(\nabla c_{1}(t-\theta)\cdot \nabla \psi^{\delta}\right)\Delta \psi^{\delta} dx\\
&-\int_{\mathbb{R}^{3}} \left(\nabla \phi \cdot\int^{\theta}_{0}\int_{\mathbb{R}^{3}} S(\theta-\widehat{\theta}, x-y) \eta_{2}(t-\widehat{\theta},y) \nabla\psi^{\delta}(\widehat{\theta},y)dyd\widehat{\theta}\right)\Delta \psi^{\delta} dx\\
&+\int_{\mathbb{R}^{3}} \left(c_{1}(t-\tau)\int^{\theta}_{0}\int_{\mathbb{R}^{3}}\nabla_{y} \Gamma(t-\widehat{\theta},y,t-\theta,x) \cdot\left(\eta_{2}(t-\widehat{\theta},y)  \nabla\psi^{\delta}(\widehat{\theta},y)\right)dyd\widehat{\theta}\right)\Delta \psi^{\delta} dx\\
&-\int_{\mathbb{R}^{3}} \left(\nabla \phi\cdot \int^{\theta}_{0}\int_{\mathbb{R}^{3}} S(\theta-\widehat{\theta},x-y)\nabla c_{2}(t-\widehat{\theta},y)\mathcal{T}(\widehat{\theta},t,y,\psi^{\delta})dyd\widehat{\theta}\right)\Delta \psi^{\delta}  dx\\
&=\text{I+II+III+IV+V}.
\end{split}
\end{equation}

\noindent
\textbullet \ \ By the divergence free condition of $v_{1}$, we have
\[
\text{I}\leq \left\|\nabla v_{1}(t-\theta)\right\|_{L^{\infty}} \left\|\nabla \psi^{\delta}\right\|^{2}_{L^{2}}.
\]

\noindent
\textbullet \ \ By the integration by parts,
\[
\text{II} \leq \left(\left\|\nabla^{2} c_{1}(t-\theta)\right\|_{L^{\infty}}+\frac{1}{2}\left\|\Delta c_{1}(t-\theta)\right\|_{L^{\infty}}\right) \left\|\nabla \psi^{\delta}\right\|^{2}_{L^{2}}.
\]

\noindent
\textbullet \ \ Also, by the integration by parts
\[
\text{III} \leq \left\|\phi\right\|^{2}_{W^{2,\infty}}\left\|\nabla \psi^{\delta}\right\|^{2}_{L^{2}} +\left\|\int^{\theta}_{0}\int_{\mathbb{R}^{3}} S(\theta-\widehat{\theta}, x-y) \eta_{2}(t-\widehat{\theta},y) \nabla\psi^{\delta}(\widehat{\theta},y)dyd\widehat{\theta}\right\|^{2}_{H^{1}}.
\]

Let
\[
U(\theta,x)=\int^{\theta}_{0}\int_{\mathbb{R}^{3}} S(\theta-\widehat{\theta}, x-y) \eta_{2}(t-\widehat{\theta},y) \nabla\psi^{\delta}(\widehat{\theta},y)dyd\widehat{\theta}.
\]
Then, $U$ satisfies the following system:
\begin{equation*}
\begin{split}
\partial_{\theta}U-\Delta U+\nabla p_{1}=\eta_{2}(t-\theta)\nabla\psi^{\delta} &\quad \text{in $\mathbb{R}^{3}_{T}$}, \\
\nabla\cdot U=0 &\quad \text{in $\mathbb{R}^{3}_{T}$},\\
U(0,x)=0 &\quad \text{in $\mathbb{R}^{3}$}.
\end{split}
\end{equation*}
for some scalar function $p_{1}$. Thus,
\begin{equation*}
\begin{split}
\frac{1}{2}\frac{d}{d\theta}\|U\|^{2}_{L^{2}}+\left\|\nabla U\right\|^{2}_{L^{2}} &\leq \|U\|_{L^{6}}\left\|\eta_{2}(t-\theta)\right\|_{L^{3}}\left\|\nabla \psi^{\delta}\right\|_{L^{2}}\\
&\leq \frac{1}{2}\left\|\nabla U\right\|^{2}_{L^{2}}+C\left\|\eta_{2}(t-\theta)\right\|^{2}_{L^{3}}\left\|\nabla\psi^{\delta}\right\|^{2}_{L^{2}}
\end{split}
\end{equation*}
and hence
\[
\frac{d}{d\theta}\|U\|^{2}_{L^{2}}+\left\|\nabla U\right\|^{2}_{L^{2}} \leq C\left\|\eta_{2}(t-\theta)\right\|^{2}_{L^{3}}\left\|\nabla\psi^{\delta}\right\|^{2}_{L^{2}}.
\]
This implies that
\begin{equation*}
\begin{split}
\int^{\theta}_{0}\text{III}d\widehat{\theta}\leq \left(\left\|\phi\right\|^{2}_{W^{2,\infty}} +C(1+t^{2})\left\|\eta_{2}\right\|^{2}_{L^{\infty}_{t}L^{3}} \right) \int^{\theta}_{0}\left\|\nabla \psi^{\delta}(\widehat{\theta})\right\|^{2}_{L^{2}}d\widehat{\theta}.
\end{split}
\end{equation*}

\noindent
\textbullet \ \ In a similar way to $\text{III}$, we estimate $\text{IV}$. First,
\begin{equation*}
\begin{split}
\text{IV} &\leq \left\|c_{1}(t-\theta)\right\|^{2}_{W^{1,\infty}}\left\|\nabla \psi^{\delta}\right\|^{2}_{L^{2}}  \\
&+ \left\|\int^{\theta}_{0}\int_{\mathbb{R}^{3}}\nabla_{y} \Gamma(t-\widehat{\theta},y,t-\theta,x) \cdot\left(\eta_{2}(t-\widehat{\theta},y)  \nabla\psi^{\delta}(\widehat{\theta},y)\right)dyd\widehat{\theta}\right\|^{2}_{H^{1}}
\end{split}
\end{equation*}
Let
\[
U(\theta,x)=\int^{\theta}_{0}\int_{\mathbb{R}^{3}}\nabla \Gamma(t-\widehat{\theta},y,t-\theta,x) \cdot\left(\eta_{2}(t-\widehat{\theta},y)  \nabla\psi^{\delta}(\widehat{\theta},y)\right)dyd\widehat{\theta}.
\]
After integration by parts, $U$ satisfies the following equation
\begin{equation*}
\begin{split}
\partial_{\theta}U+v_{1}\cdot \nabla U +\eta_{2}U-\Delta U=\nabla \cdot \left(\eta_{2}(t-\theta)\nabla \psi^{\delta}\right) &\quad \text{in $\mathbb{R}^{3}_{T}$}, \\
U(0,x)=0 &\quad \text{in $\mathbb{R}^{3}$}.
\end{split}
\end{equation*}
Then, using the divergence-free condition of $v_{1}$ and the sign condition of $\eta_{2}$,
\begin{equation*}
\begin{split}
\frac{1}{2}\frac{d}{d\theta}\|U\|^{2}_{L^{2}}+\left\|\nabla U\right\|^{2}_{L^{2}} &\leq \|\nabla U\|_{L^{2}}\left\|\eta_{2}(t-\theta)\right\|_{L^{\infty}}\left\|\nabla \psi^{\delta}\right\|_{L^{2}}\\
&\leq \frac{1}{2}\left\|\nabla U\right\|^{2}_{L^{2}}+C\left\|\eta_{2}(t-\theta)\right\|^{2}_{L^{\infty}}\left\|\nabla\psi^{\delta}\right\|^{2}_{L^{2}}
\end{split}
\end{equation*}
and hence
\[
\|U\|^{2}_{L^{\infty}_{\theta}L^{2}}+\left\|\nabla U\right\|^{2}_{L^{2}_{\theta}L^{2}}\leq C \left\|\eta_{2}\right\|^{2}_{L^{\infty}_{t}L^{\infty}} \left\|\nabla \psi^{\delta}\right\|^{2}_{L^{2}_{\theta}L^{2}}.
\]
So, we obtain
\begin{equation*}
\begin{split}
\int^{\tau}_{0}\text{IV}d\widehat{\theta} \leq \left(\left\|c_{1}\right\|^{2}_{L^{\infty}_{t}W^{1,\infty}} +(1+t^{2})\left\|\eta_{2}\right\|^{2}_{L^{\infty}_{t}L^{\infty}}\right)\int^{\theta}_{0}\left\|\nabla \psi^{\delta}(\widehat{\theta})\right\|^{2}_{L^{2}}d\widehat{\theta}.
\end{split}
\end{equation*}

\noindent
\textbullet \ \ We finally estimate $\text{V}$:
\[
\text{V} \leq \left\|\phi\right\|^{2}_{W^{2,\infty}} \left\|\nabla \psi^{\delta}\right\|^{2}_{L^{2}} + \left\|\int^{\theta}_{0}\int_{\mathbb{R}^{3}} S(\theta-\widehat{\theta},x-y)\nabla c_{2}(t-\widehat{\theta},y)\mathcal{T}(\widehat{\theta},t,y,\psi^{\delta})dyd\widehat{\theta} \right\|^{2}_{H^{1}}.
\]
Let
\[
U(\theta,x)=\int^{\theta}_{0}\int_{\mathbb{R}^{3}} S(\theta-\widehat{\theta},x-y)\nabla c_{2}(t-\widehat{\theta},y)\mathcal{T}(\widehat{\theta},t,y,\psi^{\delta})dyd\widehat{\theta}.
\]
Then, $U$ satisfies the following system:
\begin{equation*}
\begin{split}
\partial_{\theta}U-\Delta U+\nabla p_{2}=\nabla c_{2}(t-\theta)\mathcal{T}(\theta,t,y,\psi^{\delta}) &\quad \text{in $\mathbb{R}^{3}_{T}$}, \\
\nabla\cdot U=0 &\quad \text{in $\mathbb{R}^{3}_{T}$},\\
U(0,x)=0 &\quad \text{in $\mathbb{R}^{3}$}
\end{split}
\end{equation*}
for some scalar function $p_{2}$. Then,
\begin{equation*}
\begin{split}
\frac{1}{2}\frac{d}{d\theta}\|U\|^{2}_{L^{2}}+\left\|\nabla U\right\|^{2}_{L^{2}}&\leq \|U\|_{L^{6}}\left\|\nabla c_{2}(t-\theta)\right\|_{L^{3}} \left\|\mathcal{T}(\theta,t,y,\psi^{\delta})\right\|_{L^{2}}\\
&\leq \frac{1}{2}\left\|\nabla U\right\|^{2}_{L^{2}}+2\left\|\nabla c_{2}(t-\theta)\right\|^{2}_{L^{3}}\left\|\mathcal{T}(\theta,t,y,\psi^{\delta})\right\|^{2}_{L^{2}}\\
&\leq \frac{1}{2}\left\|\nabla U\right\|^{2}_{L^{2}}+C\left\|\nabla
c_{2}(t-\theta)\right\|^{2}_{L^{3}}
\left\|\eta_{2}\right\|^{2}_{L^{\infty}_{t}L^{\infty}}
\int^{\theta}_{0} \left\|\nabla
\psi^{\delta}(\widehat{\theta})\right\|^{2}_{L^{2}}d\widehat{\theta},
\end{split}
\end{equation*}
which implies that
\[
\|U\|^{2}_{L^{\infty}_{\theta}L^{2}}+\left\|\nabla U\right\|^{2}_{L^{2}_{\theta}L^{2}}\leq C \theta^{2}\left\|\nabla c_{2}\right\|^{2}_{L^{\infty}_{t}L^{3}} \left\|\eta_{2}\right\|^{2}_{L^{\infty}_{t}L^{\infty}} \left\|\nabla \psi^{\delta}\right\|^{2}_{L^{2}_{\theta}L^{2}}.
\]
Therefore, we obtain
\[
\int^{\theta}_{0} \text{V}d\widehat{\theta}  \leq \left(\left\|\phi\right\|^{2}_{W^{2,\infty}} +C(1+t^{2})\left\|\nabla c_{2}\right\|^{2}_{L^{\infty}_{t}L^{3}} \left\|\eta_{2}\right\|^{2}_{L^{\infty}_{t}L^{\infty}} \right) \int^{\theta}_{0}\left\|\nabla \psi^{\delta}(\widehat{\theta})\right\|^{2}_{L^{2}}d\widehat{\theta}.
\]

\noindent
\textbullet \ \ Collecting all terms together, we have
\begin{equation*}
\begin{split}
& \left\|\nabla \psi^{\delta}(t)\right\|^{2}_{L^{2}}+\mu\left\|\nabla \psi^{\delta}\right\|^{2}_{L^{2}_{t}L^{2}}+ \int^{t}_{0}\int\left(\delta+ A^{\alpha+1}(t-\tau)\right)\left|\Delta \psi^{\delta}\right|^{2}dxd\tau \\
&\leq \left\|\nabla \psi_{0}\right\|^{2}_{L^{2}}+\int^{t}_{0}\mathcal{J}(t)\left\|\nabla \psi^{\delta}(\tau)\right\|^{2}_{L^{2}}d\tau,
\end{split}
\end{equation*}
where
\begin{equation*}
\begin{split}
\mathcal{J}(t)&=\left\|\phi\right\|^{2}_{W^{2,\infty}}+ \left\|\nabla v_{1}\right\|_{L^{\infty}_{t}L^{\infty}}+ \left\|\nabla^{2} c_{1}\right\|_{L^{\infty}_{t}L^{\infty}}+\left\|\Delta c_{1}\right\|_{L^{\infty}_{t}L^{\infty}}\\
& +C(1+t^{2})\left\|\eta_{2}\right\|^{2}_{L^{\infty}_{t}(L^{3}\cap L^{\infty})}+\left\|c_{1}\right\|^{2}_{L^{\infty}_{t}W^{1,\infty}}+  C(1+t^{2})\left\|\nabla c_{2}\right\|^{2}_{L^{\infty}_{t}L^{3}} \left\|\eta_{2}\right\|^{2}_{L^{\infty}_{t}L^{\infty}}.
\end{split}
\end{equation*}

We note that H\"older regularities in \cite{Chung 2} are enough to control $\mathcal{J}(t)$. Hence, we derive the following a prior estimate by Gronwall's inequality
\[
\sup_{0\leq t\leq T}\left\|\nabla \psi^{\delta}(t)\right\|^{2}_{L^{2}}\leq\left\|\nabla \psi_{0}\right\|^{2}_{L^{2}}e^{2T\mathcal{J}(T)}.
\]
Therefore,
\[
\sup_{0\leq t\leq T}\left\|\nabla \psi^{\delta}(t)\right\|^{2}_{L^{2}}\leq C(T, \mathcal{J}(T))\left\|\nabla \psi_{0}\right\|^{2}_{L^{2}}
\]
and thus
\[
\delta \int^{T}_{0}\left|\Delta \psi^{\delta}\right|^{2}dxd\tau\leq C(T,\mathcal{J}(T))\left\|\nabla \psi_{0}\right\|^{2}_{L^{2}}.
\]
This implies that
\[
\lim_{\delta\rightarrow 0}\delta \int^{t}_{0} \int_{\mathbb{R}^{3}}\chi(\tau,x) \Delta \psi^{\delta}(\tau,x)dxd\tau=0
\]
for a.e. $0<t<T$ and for all $\chi\in L^{2}_{t}L^{2}$. This completes the proof of Lemma \ref{main lemma}.

\begin{remark}\label{comment-100}
The vanishing viscosity argument does not work with $q>1$ of the system \eqref{KS}. Indeed, when we estimate the right-hand side of (\ref{a priori vanishing viscosity}), we need to perform the integration by parts to move one derivative in $\Delta \zeta$ to the other terms in $\text{II}$. If $q>1$, $A^{1+\alpha}$ appears in $\text{II}$ and it requires that $\eta\in W^{1,\infty}(\mathbb{R}^{3}_{T})$, which is beyond the regularity in \cite{Chung 2}. Developing new methods dealing with \eqref{KS} for the case $q>1$ and possibly other equations, having solutions but no uniqueness, will be some of our next study.
\end{remark}

\section{Fundamental solution of a parabolic equation}
We here present the proof of Theorem \ref{theorem c}. In fact, Theorem \ref{theorem c} is a corollary of a more general statement, Proposition \ref{thm_fs01}. In this section, we shall denote by $\sL$ a parabolic operator on $\bR^n_T=(0,T) \times \bR^n$ ($n \ge 1$) defined by
\[
\sL u = u_t - \dv(A \nabla u) + a \cdot \nabla u + b u.
\]
The coefficients $A=A(t,x)$ is an $n \times n$ (not necessarily symmetric) matrix with entries $a^{ij}(t,x)$ that is uniformly parabolic and bounded. More precisely, we assume
\begin{equation}            \label{ellipticity}
\lambda \abs{\xi}^2 \le A(t,x) \xi \cdot \xi, \quad \Abs{ A(t,x)\xi \cdot \eta} \le \Lambda \abs{\xi} \abs{\eta},\qquad \forall \xi, \eta \in \bR^n,\;\;\forall (t,x) \in \bR^n_T
\end{equation}
for some positive constants $\lambda$ and $\Lambda$.
We assume that $a=(a^1,\ldots, a^n)$ is a divergence-free vector field. Then the adjoint operator $\sL^\ast$ of $\sL$ is given by
\[
\sL^\ast u = -u_t - \dv\left(A\tran \nabla u\right) - a \cdot \nabla u + b u.
\]

We define the parabolic distance between the points $X=(t,x)$ and $Y=(s,y)$ in $\bR^{n+1}$ as
\[
\abs{X-Y}_p:=\max(\sqrt{\abs{t-s}},\abs{x-y}).
\]
We use the following notations for basic cylinders in $\bR^{n+1}$:
\begin{align*}
Q^-_r(X)&=(t-r^2,t)\times B_r(x),\\
Q^+_r(X)&=(t,t+r^2)\times B_r(x),\\
Q_r(X)&=(t-r^2,t+r^2)\times B_r(x).
\end{align*}

\begin{proposition}\label{thm_fs01}
Let $a=(a^1,\ldots, a^n)$ be a divergence-free vector field and $b$ be a nonnegative function on $\bR^n_T:=(0,T) \times \bR^n$. Assume that $a$ and $b$ are bounded on $\bR^n_T$. Then, there exists a unique fundamental solution $\Gamma(t,x,s,y)$ of $\sL$ on $\bR^n_T$ which satisfies the following pointwise bound:
\begin{equation}\label{gaussian01}
\abs{\Gamma(t,x,s,y)} \le C (t-s)^{-\frac{n}{2}} \exp\left(-c\frac{\abs{x-y}^2}{t-s}\right),\quad (0< s <t < T),
\end{equation}
where $c=\frac{\lambda}{4\Lambda^2}-\epsilon$ for any $0<\epsilon< \frac{\lambda}{4\Lambda^2}$ and $C=C(n, \lambda, \Lambda, T, \norm{a}_{L^\infty(\bR^n_T)} ,\norm{b}_{L^\infty(\bR^n_T)}, \epsilon)$.
\end{proposition}

\subsection*{Proof of Proposition \ref{thm_fs01}}
To prove this proposition, we closely follow methods used in \cite{CDK08}.
We recall some notations introduced there.
For $U \subset\bR^{n+1}$, we write $U(t_0)$ for the set of all points $(t_0,x)$ in $U$ and $I(U)$ for the set of all $t$ such that $U(t)$ is nonempty.
We denote
\[
\tri{u}_{U}^2= \norm{\nabla u}_{L^{2}(U)}^2 +\esssup\limits_{t\in I(U)} \norm{u(\cdot, t)}_{L^2(U(t))}^2.
\]
We ask reader to consult \cite{CDK08} for definition of functions spaces such as $\mathring{V}^{1,0}_2(Q)$.
We first note that if $u \in \mathring{V}^{1,0}_2(\bR^n_T)$ is the weak solution of the problem
\[
\left\{
\begin{aligned}
\sL u = f &\quad \text{ in }\;(t_0,t_1) \times \bR^n,\\
u=\psi &\quad \text{ on }\;\set{t_0}\times \bR^n,
\end{aligned}
\right.
\]
then $u$ satisfies the energy inequality
\begin{equation}\label{eq1118502}
\begin{split}
\sup_{t_0\le t \le t_1} \int_{\bR^n} \abs{u(t,x)}^2\,dx + &\lambda \int_{t_0}^{t_1}\!\!\!\int_{\bR^n} \abs{\nabla u(t,x)}^2\,dxdt \\
&\le \int_{\bR^n} \abs{\psi(x)}^2 dx+ C \left( \int_{t_0}^{t_1}\!\!\! \int_{\bR^n} \abs{f(t,x)}^{\frac{2(n+2)}{n+4}}\,dxdt \right)^{\frac{n+4}{n+2}},
\end{split}
\end{equation}
where $C=C(n, \lambda, \Lambda)$. Here, we essentially use the assumption that $\nabla\cdot a=0$ and $b\ge 0$. A similar statement is true for an adjoint problem.

We construct an approximate fundamental solution as follows.
For $Y=(s,y)\in \bR^n_T$ denote
\[
d_Y:=\sqrt{\max(s, T-s)}
\]
so that $X \in \bR^n_T$ when $\abs{X-Y}_p<d_Y$.
For $0<\epsilon <d_Y$, let $v_\epsilon \in \mathring{V}^{1,0}_2(\bR^n_T)$ be a unique weak solution of the problem
\begin{equation}        \label{mn.eq1}
\left\{\begin{array}{l l}
\sL u=\frac{1}{\abs{Q^{-}_\epsilon}}1_{Q^-_\epsilon(Y)}\\
u(0,\cdot)= 0.\end{array}\right.
\end{equation}
With the aid of the energy inequality \eqref{eq1118502}, the unique solvability of the problem \eqref{mn.eq1} in $\mathring{V}^{1,0}_2(\bR^n_T)$ follows from the Galerkin method described  in \cite[\S III.5]{LSU}.
Observe that the energy inequality \eqref{eq1118502} implies
\begin{equation}        \label{mn.eq1c}
\tri{v_\epsilon}_{\bR^n_T} \le C \epsilon^{-\frac{n}{2}}.
\end{equation}
Next, for a given $F \in C^\infty_c(\bR^n_T)$, let $u \in \mathring{V}^{1,0}_2(\bR^n_T)$ be the weak solution of the \emph{backward problem}
\begin{equation}        \label{mn.eq2}
\left\{\begin{array}{l l}
\sL^\ast u=F\\
u(T,\cdot)= 0.\end{array}\right.
\end{equation}
Then, we have the identity
\begin{equation}        \label{mn.eq3}
\int_{\bR^n_T} v_\epsilon F \,dxdt  =\fint_{Q^-_\epsilon(Y)} u \,dxdt.
\end{equation}
If we assume that $F$ is supported in $Q^{+}_R(X_0) \subset \bR^n_T$, then an inequality similar to \eqref{eq1118502} yields
\begin{equation}        \label{mn.eq4b}
\tri{u}_{\bR^n_T} \le C \norm{F}_{L^{2(n+2)/(n+4)}(Q^{+}_R(X_0))}.
\end{equation}
By the well-known embedding theorem (see e.g., \cite[\S II.3]{LSU}), we have
\[
\norm{u}_{L^{2(n+2)/n}(\bR^n_T)} \le C(n) \tri{u}_{\bR^{n}_T}.
\]
By combining the above two inequalities and using H\"older's inequality, we get
\begin{equation}\label{eq15.54m}
\norm{u}_{L^2(Q^+_R(X_0))}\le C(n,\lambda, \Lambda) R\norm{F}_{L^{2(n+2)/(n+4)}(Q^{+}_R(X_0))}.
\end{equation}

\begin{lemma}   \label{lem01sk}
Let $u$ be a weak solution of
\[
\sL u = F \;\text{ in }\; Q_R^{-}(X_0),
\]
where $Q_R^{-}(X_0) \subset \bR^n_T$.
Then $u$ is locally H\"older continuous in $Q_R^{-}(X_0)$. In particular, $u$ is locally bounded in $Q_R^{-}(X_0)$ and for any $p>0$, we have the estimate
\[
\sup_{Q_{R/2}^{-}(X_0)} \abs{u}  \le C\left( \fint_{Q_R^{-}(X_0)} \abs{u}^p dxdt\right)^{\frac{1}{p}} + C R^2 \norm{F}_{L^\infty(Q_R^{-}(X_0))},
\]
where $C=C(n, p, \lambda, \Lambda, \norm{a}_{L^\infty(\bR^n_T)}, \norm{b}_{L^\infty(\bR^n_T)})$.
A similar statement is true for a weak solution of
\[
\sL^\ast u = F \;\text{ in }\; Q_R^{+}(X_0),
\]
where $Q_R^{+}(X_0) \subset \bR^n_T$.
\end{lemma}

\begin{proof}
See \cite[\S III.8 and \S III.10]{LSU}.
\end{proof}

By utilizing (\ref{eq15.54m}) and the above lemma, we get
\begin{equation}        \label{mn.eq4d}
\sup_{Q^{+}_{R/2}(X_0)} \abs{u} \le CR^{2}\norm{F}_{L^\infty(Q^{+}_R(X_0))}.
\end{equation}
If $Q^-_\epsilon(Y)\subset Q^+_{R/2}(X_0)$, then by \eqref{mn.eq3} and \eqref{mn.eq4d}, we obtain
\[
\Abs{\int_{Q^+_R(X_0)}  v_\epsilon F\,} \le \fint_{Q^-_\epsilon(Y)}\abs{u}\le CR^{2} \norm{F}_{L^{\infty}(Q^+_R(X_0))}.
\]
Therefore, by duality, it follows that we have
\begin{equation}        \label{mn.eq4e}
\norm{v_\epsilon}_{L^1(Q_R^+(X_0))}\le CR^2.
\end{equation}
We define the \emph{averaged fundamental solution} $\Gamma^\epsilon(X, Y)=\Gamma^\epsilon(t, x, s, y)$ for $\sL$ by setting
\[
\Gamma^\epsilon(\cdot, Y)=v_\epsilon.
\]

\begin{lemma}   \label{lem02sk}
Let $X=(t,x)$, $Y=(s,y) \in \bR^n_T$ and assume $0<\abs{X-Y}_p < \frac16 d_Y$.
Then
\begin{equation}        \label{mn.eq4g}
\abs{\Gamma^{\epsilon}(X, Y)} \le C \abs{X-Y}_p^{-n}, \quad \forall \epsilon < \frac13 \abs{X-Y}_p,
\end{equation}
where $C=C(n, \lambda, \Lambda, \norm{a}_{L^\infty(\bR^n_T)}, \norm{b}_{L^\infty(\bR^n_T)})$.
\end{lemma}
\begin{proof}
Denote $d=\abs{X-Y}_p$, and let $r=\frac13 d$, $X_0=(y, s-4d^2)$, and $R=6d$.
It is easy to see that for $\epsilon<\frac13 d$, we have
\[
Q^{-}_\epsilon(Y)\subset Q^{+}_{R/2}(X_0),\quad Q^{-}_r(X)\subset Q^{+}_{R}(X_0).
\]
and also that $v_\epsilon$ satisfies $\sL v_\epsilon=0$ in $Q^{-}_r(X)$.
Then, by Lemma~\ref{lem01sk}, we have
\[
\norm{v_\epsilon}_{L^\infty(Q_{r/2}^{-}(X))}\le \frac{C}{r^{n+2}}\norm{v_\epsilon}_{L^1(Q^{-}_r(X))}.
\]
Therefore, by \eqref{mn.eq4e}, we have $\abs{v_\epsilon(X)}\le Cr^{-n}$, which implies \eqref{mn.eq4g}.
\end{proof}

For $\epsilon < r < R< \frac16 d_Y$, let $\eta: \bR^{n+1} \to \bR$ be a smooth function such that
\begin{equation}                \label{eq_eta}
0\le \eta \le 1,\quad \eta \equiv 0\;\text{ on }\; Q_{r}(Y),\quad \eta \equiv 1\; \text{ on }\;Q_{R}(Y)^c,\quad \abs{\nabla \eta}^2 + \abs{\nabla ^2\eta}+ \abs{\eta_t}\le \tfrac{12}{(R-r)^2}.
\end{equation}
Recall that $v_\epsilon$ satisfies \eqref{mn.eq1}.
By testing with $\eta^2 v_\epsilon$ and using assumption $\nabla \cdot a=0$, we have
\begin{multline*}
0 = \int_{\bR^n} \frac12 (\eta^2 v_\epsilon^2)_t -\int_{\bR^n} \eta \eta_t v_\epsilon^2+\int_{\bR^n} \eta^2 A \nabla v_\epsilon \cdot \nabla v_\epsilon \\
+ \int_{\bR^n} 2 \eta  A \nabla v_\epsilon \cdot \nabla \eta v_\epsilon -\int_{\bR^n}  a\cdot \nabla \eta  \eta v_\epsilon^2 + \int_{\bR^n} \eta^2 b v_\epsilon^2.
\end{multline*}
Then by using \eqref{ellipticity} and $b \ge 0$, we get
\[
\int_{\bR^n} \frac12 (\eta^2 v_\epsilon^2)_t + \lambda \int_{\bR^n} \eta^2 \abs{\nabla v_\epsilon}^2
\le  \int_{\bR^n} \eta \abs{\eta_t}\, \abs{v_\epsilon}^2 + 2\Lambda\int_{\bR^n} \eta \abs{\nabla  v_\epsilon}\,\abs{\nabla \eta}\,\abs{v_\epsilon}+  \int_{\bR^n}  \eta \abs{\nabla \eta} \,\abs{a} v_\epsilon^2.
\]
By using Young's inequality and integrating over $t$, we get
\[
\sup_{0\le t \le T} \, \frac12  \int_{\bR^n} \eta^2 v_\epsilon^2+ \frac{\lambda}{2} \int_{\bR^n_T} \eta^2 \abs{\nabla v_\epsilon}^2
\le \int_{\bR^n_T}  \left\{\abs{\eta_t}+\frac{2\Lambda^2}{\lambda} \abs{\nabla \eta}^2 +\abs{a} \,\abs{\nabla \eta}\right\}  v_\epsilon^2.
\]
Therefore, by \eqref{eq_eta} and noting that $R-r< \sqrt{T}$ and so  
\[
|\nabla \eta| \le \sqrt{12}/(R-r) = \sqrt{12}(R-r)/(R-r)^2 \le \sqrt{12  T}/(R-r)^2,
\]
we have
\begin{equation}            \label{eq13.15th}
\sup_{0\le t \le T} \,  \int_{\bR^n} \eta^2 v_\epsilon^2+ \int_{\bR^n_T} \eta^2 \abs{\nabla v_\epsilon}^2 \le \frac{C}{(R-r)^2} \int_{Q_R(Y) \setminus Q_r(Y)} v_\epsilon^2,
\end{equation}
where $C=C(n, \lambda, \Lambda, T, \norm{a}_{L^\infty(\bR^n_T)})$.

Then by setting $r=\frac12 R$ in \eqref{eq13.15th} and applying Lemma~\ref{lem02sk}, we obtain
\[
\tri{\Gamma^\epsilon(\cdot, Y)}_{\bR^n_T \setminus Q_R(Y)}^2 \le C
R^{-2} \int_{\set{R/2<\abs{X-Y}_p<R}} \abs{X-Y}_p^{-2n}\,dX  \le
CR^{-n},
\]
whenever $\epsilon < \frac16 R$ and $R<\frac16 d_Y$.
On the other hand, in the case when $\epsilon \ge \frac16 R$, \eqref{mn.eq1c} yields
\[
\tri{\Gamma^\epsilon(\cdot, Y)}_{\bR^n_T \setminus Q_R(Y)}^2 \le \tri{\Gamma^\epsilon(\cdot, Y)}_{\bR^n_T}^2 \le C \epsilon^{-n} \le C R^{-n}.
\]
Therefore, we have for all $\epsilon<d_Y$ and $R<\frac16 d_Y$
\begin{equation}                \label{eq13.35th}
\tri{\Gamma^\epsilon(\cdot, Y)}_{\bR^n_T \setminus Q_R(Y)}^2 \le CR^{-n}.
\end{equation}

We note that the estimate \eqref{eq13.35th} corresponds to \cite[(4.6)]{CDK08}.
With Lemma~\ref{lem01sk} at hand, we can repeat the same argument as in \cite{CDK08} to construct the fundamental solution
\[
\Gamma(X, Y)=\Gamma(t, x,s, y)
\]
for $\sL$ in $\bR^n_T$. See Section 4.2 -- 4.3 of \cite{CDK08} for details.

Next, we show that $\Gamma(x,t,y,s)$ satisfies the Gaussian bound \eqref{gaussian01}. We modify an argument in \cite{HK04}, which is in turn based on Davies \cite{Davies} and Fabes-Stroock \cite{FS}.
Let $\psi$ be a bounded Lipschitz function on $\bR^n$ satisfying $\abs{\nabla \psi} \le \gamma$ a.e. for some $\gamma>0$ to be chosen later.
Fix $s \in (0,T)$.
For $s<t<T$, we define an operator $P^\psi_{s\to t}$ on  $L^2(\bR^n)$ as follows. For a given $f\in L^2(\bR^n)$, let $u \in \mathring{V}^{1,0}_2((s,T)\times \bR^n)$ be the weak solution of the problem
\[
\left\{\begin{array}{l l}
\sL  u=0\\
u(s,\cdot)= e^{-\psi} f.\end{array}\right.
\]
Then, we define
\[
P^\psi_{s\to t} f(x):= e^{\psi(x)} u(t,x)
\]
so that
\[
P^\psi_{s\to t} f(x)=
e^{\psi(x)}\int_{\bR^n} \Gamma(t, x,s,y)e^{-\psi(y)}f(y)\,dy.
\]

Let us denote
\[
I(t):=\int_{\bR^n} e^{2\psi(x)} u(t, x)^2\,dx.
\]
Then, $I'(t)$ satisfies for a.e. $t \in (s, T)$ that
\begin{align*}
I'(t) &=-2\int_{\bR^n} A \nabla u \cdot \nabla (e^{2\psi} u )+ e^{2\psi} u\,a \cdot \nabla u + e^{2\psi} b u^2\,dx \\
& \le -2\int_{\bR^n} e^{2\psi} A \nabla u \cdot \nabla u\,dx -4\int_{\bR^n} e^{2\psi} u\, A \nabla u \cdot \nabla\psi\,dx -2\int_{\bR^n} a \cdot \nabla u \,e^{2\psi} u\,dx,
 \end{align*}
where we used that $b \ge 0$.

By using $\nabla \cdot a=0$, we find that
\[
0=\int_{\bR^n} a \cdot \nabla(e^\psi u)\, e^\psi u\,dx = \int_{\bR^n} a \cdot \nabla u \,e^{2\psi} u \,dx+ \int_{\bR^n} a \cdot \nabla \psi \, e^{2\psi} u^2\,dx.
\]
For the rest of proof, we set
\[
\kappa:=\norm{a}_{L^\infty(\bR^n_T)}.
\]
By using $\abs{\nabla \psi} \le \gamma$, we then obtain
\begin{align*}
I'(t) & \le  -2 \lambda \int_{\bR^n} e^{2\psi} \abs{\nabla u}^2 \,dx + 4 \Lambda \gamma \int_{\bR^n} e^{\psi} \abs{u}\,e^{\psi} \abs{\nabla u} \,dx  + 2 \gamma \kappa \int_{\bR^n} e^{2\psi} u^2\,dx \\
& \le  \left(2 \Lambda^2 \lambda^{-1} \gamma^2 + 2 \kappa \gamma \right) \int_{\bR^n} e^{2\psi} u^2\,dx = 2(\Lambda^2 \lambda^{-1} \gamma^2 + \kappa \gamma) \,I(t).
\end{align*}
The initial condition $u(s, \cdot)= e^{-\psi}f$ yields
\[
I(t)\le  e^{2 (\Lambda^2 \lambda^{-1} \gamma^2 + \kappa \gamma) (t-s)}\norm{f}_{L^2(\bR^n)}^2.
\]
Since $I(t)=\norm{P^\psi_{s\to t} f}_{L^2(\bR^n)}^2$, we have derived
\begin{equation}                \label{eq15.30e}
\norm{P^\psi_{s\to t} f}_{L^2(\bR^n)} \le e^{(\Lambda^2 \lambda^{-1} \gamma^2 + \kappa \gamma) (t-s)}\norm{f}_{L^2(\bR^n)}.
\end{equation}

By Lemma~\ref{lem01sk}, we estimate
\begin{align*}
e^{-2\psi(x)} \abs{P^\psi_{s\to t}f(x)}^2 &= \abs{u(x,t)}^2\\
&\le \frac{C}{(t-s)^{\frac{n+2}{2}}} \int_s^t\int_{B_{\sqrt{t-s}}(x)}\abs{u(y,\tau)}^2\,dy\,d\tau\\
&\le \frac{C}{(t-s)^{\frac{n+2}{2}}} \int_s^t\int_{B_{\sqrt{t-s}}(x)}e^{-2\psi(y)}
 \abs{P^\psi_{s\to\tau}f(y)}^2\,dy\,d\tau.
\end{align*}
Hence, by using \eqref{eq15.30e} we find
\begin{align*}
\abs{P^\psi_{s\to t}f(x)}^2 &\le C(t-s)^{-\frac{n+2}{2}}
\int_s^t\int_{B_{\sqrt{t-s}}(x)}e^{2\psi(x)-2\psi(y)} \abs{P^\psi_{s\to\tau} f(y)}^2\,dy\,d\tau\\
&\le C(t-s)^{-\frac{n+2}{2}} \int_s^t\int_{B_{\sqrt{t-s}}(x)}e^{2\gamma\sqrt{t-s}}
\,\abs{P^\psi_{s\to\tau} f(y)}^2\,dy\,d\tau\\
&\le C (t-s)^{-\frac{n+2}{2}}\,e^{2\gamma\sqrt{t-s}} \int_s^t e^{2(\Lambda^2 \lambda^{-1} \gamma^2 + \kappa \gamma) (\tau-s)}\norm{f}_{L^2(\bR^n)}^2\,d\tau\\
&\le C (t-s)^{-\frac{n}{2}}\,e^{2\gamma\sqrt{t-s}+2(\Lambda^2 \lambda^{-1} \gamma^2 + \kappa \gamma) (t-s)} \norm{f}_{L^2(\bR^n)}^2.
\end{align*}
We have thus derived the following $L^2\to L^\infty$ estimate for $P^\psi_{s\to t}$:
\begin{equation}                \label{eq16.40e}
\norm{P^\psi_{s\to t} f}_{L^\infty(\bR^n)} \le  C (t-s)^{-\frac{n}{4}}\,e^{\gamma\sqrt{t-s}+(\Lambda^2 \lambda^{-1} \gamma^2 + \kappa \gamma) (t-s)} \norm{f}_{L^2(\bR^n)}.
\end{equation}

We also define the operator $Q^\psi_{t\to s}$ on $L^2(\bR^n)$  by
\[
Q^\psi_{t\to s}g(y)= e^{-\psi(y)}v(y,s),
\]
where $v$ is the weak solution of the backward problem
\[
\left\{\begin{array}{l l}
\sL^\ast  u=0\\
u(t,\cdot)= e^{\psi} g.\end{array}\right.
\]
Then, by a similar calculation, we get
\[
\norm{Q^\psi_{t\to s} g}_{L^\infty(\bR^n)} \le  C (t-s)^{-\frac{n}{4}}\,e^{\gamma\sqrt{t-s}+(\Lambda^2 \lambda^{-1} \gamma^2 + \kappa \gamma) (t-s)}
\norm{g}_{L^2(\bR^n)}.
\]

Since (see \cite[Section 5.1]{CDK08})
\[
\int_{\bR^n}(P^\psi_{s\to t} f) \, gdx=  \int_{\bR^n} f\,(Q^\psi_{t\to s} g)dx,
\]
by duality, for any $f \in C^\infty_c(\bR^n)$, we have
\begin{equation}                \label{eq16.41e}
\norm{P^\psi_{s\to t} f}_{L^2(\bR^n)} \le  C (t-s)^{-\frac{n}{4}}\,e^{\gamma\sqrt{t-s}+(\Lambda^2 \lambda^{-1} \gamma^2 + \kappa \gamma) (t-s)} \norm{f}_{L^1(\bR^n)}.
\end{equation}

Now,  set $r=\frac{s+t}{2}$ and observe that by the uniqueness, we have
\[
P^\psi_{s\to t}f= P^\psi_{r\to t}(P^\psi_{s\to r}f),
\quad \forall f\in C^\infty_c(\bR^n).
\]
Then, by noting that
\[
t-r=r-s=\frac12 (t-s),
\]
we obtain from \eqref{eq16.40e} and \eqref{eq16.41e} that for any $f\in C^\infty_c(\bR^n)$, we have
\[
\norm{P^\psi_{s\to t} f}_{L^\infty(\bR^n)} \le  C (t-s)^{-\frac{n}{2}}\,e^{\gamma\sqrt{2(t-s)}+(\Lambda^2 \lambda^{-1} \gamma^2 + \kappa \gamma) (t-s)} \norm{f}_{L^1(\bR^n)}.
\]

For fixed $x, y\in \bR^n$ with $x\neq y$, the above estimate imply, by duality, that
\begin{equation}                \label{eq16.51e}
e^{\psi(x)-\psi(y)}\,\abs{\Gamma(t,x,y,s)} \le C (t-s)^{-\frac{n}{2}} \,e^{\gamma\sqrt{2(t-s)}+(\Lambda^2 \lambda^{-1} \gamma^2 + \kappa \gamma) (t-s)}.
\end{equation}

We now set
\[
\psi(z)=\gamma \min( \abs{z-y}, \abs{x-y}) \quad\text{and}\quad \gamma= \frac{\lambda}{2\Lambda^2}\,\frac{\abs{x-y}}{t-s}.
\]

It should be clear that $\psi$ is a bounded Lipschitz function satisfying $\abs{\nabla\psi} \le \gamma$ a.e.,
\[
\psi(x)=\gamma \abs{x-y}=\frac{\lambda}{2\Lambda^2}\,\frac{\abs{x-y}^2}{t-s},\quad \text{and} \quad \psi(y)=0.
\]
Therefore, \eqref{eq16.51e} yields
\begin{align*}
\abs{\Gamma(t, x, s,y)} &\le C(t-s)^{-\frac{n}{2}} \exp \left\{ \frac{\lambda}{\sqrt{2}\Lambda^2}\, \frac{\abs{x-y}}{\sqrt{t-s}} + \frac{\lambda}{4\Lambda^2} \,\frac{\abs{x-y}^2}{t-s}+ \frac{\kappa \lambda}{2\Lambda^2}\, \abs{x-y} -\frac{\lambda}{2\Lambda^2}\, \frac{\abs{x-y}^2}{t-s} \right\} \\
&\le C(t-s)^{-\frac{n}{2}}  \exp \left\{ \left(\frac{\lambda}{\sqrt2\Lambda^2} + \frac{\kappa \lambda \sqrt{T}}{2\Lambda^2} \right) \frac{\abs{x-y}}{\sqrt{t-s}} - \frac{\lambda}{4\Lambda^2}\, \frac{\abs{x-y}^2}{t-s} \right\}.
\end{align*}
Note that for any $\epsilon \in (0, \lambda/ 4\Lambda^2)$, there exists a number $N=N(\lambda, \Lambda, T, \kappa, \epsilon)$ such that
\[
e^{\left(\frac{\lambda}{\sqrt2\Lambda^2} + \frac{\kappa \lambda \sqrt{T}}{2\Lambda^2} \right) r - \frac{\lambda}{4\Lambda^2} r^2} \le N e^{ -\left(\frac{\lambda}{4\Lambda^2}-\epsilon \right)r^2}, \quad\forall r \ge 0,
\]
and recall that $\kappa=\norm{a}_{L^\infty(\bR^n_T)}$. Therefore, we obtain the Gaussian bound \eqref{gaussian01}.
\qed

\subsection*{Proof of Theorem \ref{theorem c}}
To prove Theorem \ref{theorem c}, we begin with the following lemma.

\begin{lemma}   \label{lem03sk}
Assume the hypothesis of Theorem \ref{theorem c} holds. Let $u$ be a solution of
\[
\cL u = 0 \;\text{ in }\; Q_r^{-}(X_0),
\]
where $Q_r^{-}(X_0) \subset \bR^3_T$. Then $u$, $\nabla u$, $u_t$, $\nabla^2 u$ are locally H\"older continuous in $Q_r^{-}(X_0)$. In particular, we have the following estimate:
\[
r \norm{\nabla u}_{L^\infty(Q_{r/2}^{-}(X_0))} + r^2 \norm{\nabla^2 u}_{L^\infty(Q_{r/2}^{-}(X_0))} +r^2 \norm{u_t}_{L^\infty(Q_{r/2}^{-}(X_0))}  \le C \norm{u}_{L^\infty(Q_{r}^{-}(X_0))},
\]
where $C=C\bigl(\norm{a}_{C^{\beta, \beta/2}(\bR^3_T)}, \norm{b}_{C^{\alpha, \alpha/2}(\bR^3_T)}\bigr)$.
\end{lemma}
\begin{proof}
It  is a consequence of the standard Schauder theory.
\end{proof}

We apply Proposition~\ref{thm_fs01} to the operator $\cL$ and construct the fundamental solution $\Gamma(t,x,s,y)$.
Note that we have $\lambda=\Lambda=1$ in this case so that we have $\frac{\lambda}{4\Lambda^2}=\frac14$, and thus \eqref{gaussian01} follows.

Next, to prove \eqref{gaussian031}, let $r:=\frac12 \sqrt{t-s}$ and note that $u(\cdot):=\Gamma(\cdot, Y)$ satisfies
\[
\cL u=0 \quad \text{in $Q_r^{-}=Q_r^{-}(X)$}.
\]
Then, by Lemma~\ref{lem03sk}, we have
\begin{equation} \label{grad of Gamma}
\begin{split}
& \sqrt{t-s}\, \abs{ \nabla_x \Gamma(t,x,s,y)} + (t-s)\, \abs{ \nabla_x^2 \Gamma(t,x,s,y)} + (t-s)\, \abs{ \partial_t \Gamma(t,x,s,y)} \\
&\le  C \norm{u}_{L^\infty(Q_{r}^{-}(X))}.
\end{split}
\end{equation}
Note that for $Z=(\tau, z) \in Q_r^{-}(X)$, we have
\[
\textstyle
\frac34 (t-s) \le \tau-s \le t-s\quad\text{and}\quad \abs{x-y} - \frac12 \sqrt{t-s} \le \abs{z-y} \le \abs{x-y} + \frac12 \sqrt{t-s}.
\]
Therefore, by \eqref{gaussian01}, we have
\begin{equation}            \label{u bound Q}
\norm{u}_{L^\infty(Q_r^{-}(X))} \le C (t-s)^{-\frac{3}{2}} e^{-c' \frac{\abs{x-y}^2}{t-s}},
\end{equation}
where $c'>0$ is an absolute constant and $C=C\bigl(T,
\norm{a}_{C^{\beta, \beta/2}(\bR^3_T)}, \norm{b}_{C^{\beta,
\beta/2}(\bR^3_T)}\bigr)$. By (\ref{grad of Gamma}) and (\ref{u
bound Q}), we obtain the estimate \eqref{gaussian031}, which
complete the proof of Theorem \ref{theorem c}.

\vspace{5ex}

\section*{Acknowledgments}
H. Bae was supported by NRF-2015R1D1A1A01058892. 

K. Kang was supported by NRF-2017R1A2B4006484 and NRF-20151009350. 

S. Kim was supported by NRF-20151009350.

\bibliographystyle{amsplain}

\end{document}